\documentclass[a4paper,11pt]{scrartcl}

\usepackage[utf8]{inputenc}
\usepackage[hidelinks]{hyperref}
\usepackage{geometry}
\usepackage{mathtools}
\usepackage{amsfonts}
\usepackage{amssymb}
\usepackage{amsthm}
\usepackage{csquotes}
\usepackage{tikz-cd}
\usepackage[title]{appendix}

\setkomafont{section}{\large}

\newtheoremstyle{style}
{3.25ex} 
{3.25ex} 
{} 
{} 
{\bfseries\sffamily} 
{} 
{.5em} 
{} 

\newtheoremstyle{style*}
{3.25ex} 
{3.25ex} 
{} 
{} 
{} 
{} 
{0em} 
{} 

\theoremstyle{style}
\newtheorem{example}[subsection]{Example}
\newtheorem{definition}[subsection]{Definition}
\newtheorem{remark}[subsection]{Remark}

\newtheorem{notation}[subsection]{Notation}

\newtheorem{subexample}[subsubsection]{Example}
\newtheorem{subdefinition}[subsubsection]{Definition}
\newtheorem{subremark}[subsubsection]{Remark}
\newtheorem{sublemma}[subsubsection]{Lemma}
\newtheorem{subproposition}[subsubsection]{Proposition}
\newtheorem{subtheorem}[subsubsection]{Theorem}
\newtheorem{subnotation}[subsubsection]{Notation}
\newtheorem{subcorollary}[subsubsection]{Corollary}
\newtheorem{subconstruction}[subsubsection]{Construction}
\newtheorem{subvariant}[subsubsection]{Variant}
\newtheorem*{thmA}{Theorem A}
\newtheorem*{thmB}{Theorem B}

\theoremstyle{style*}
\newtheorem*{intermediate}{}

\newcommand\defeq{\mathrel{\vcenter{\baselineskip0.5ex \lineskiplimit0pt\hbox{\scriptsize.}\hbox{\scriptsize.}}}=}
\newcommand\eqdef{=\mathrel{\vcenter{\baselineskip0.5ex \lineskiplimit0pt\hbox{\scriptsize.}\hbox{\scriptsize.}}}}

\makeatletter
\def\namedlabel#1#2{\begingroup\def\@currentlabel{#2}\label{#1}\endgroup}
\makeatother

\title{\vspace{-0.65cm}\Large Weight filtrations via slopes}
\author{\large Carolyn Echter \thanks{Fakultät für Mathematik, Universität Regensburg, 93040 Regensburg, Germany \\ Email address: \href{mailto:carolyn.echter@ur.de}{\nolinkurl{carolyn.echter@ur.de}}}}
\date{\normalsize June 1, 2026}

\begin{document}

\maketitle
\vspace{-1cm}

\begin{abstract} \noindent
Mixed structures and their weight filtrations appear in various contexts, prominently Hodge theory and the theory of Galois representations. In the setting of André's formalisation of slopes, we propose an abstract framework explaining why mixed structures have shared characteristics.
\end{abstract}

\section{Introduction}

Weight filtrations have their origin in \cite{ARTICLE:2, ARTICLE:8}: While classical Hodge theory yields pure Hodge structures on the cohomology groups of smooth projective complex varieties, Deligne constructed a Hodge and a weight filtration on the cohomology of an arbitrary complex variety such that the Hodge filtration induces pure Hodge structures on the graded pieces of the weight filtration. He called the resulting filtered objects mixed Hodge structures. Simpson \cite{MISC:3} later introduced mixed twistor structures as a generalisation. In his work \cite{ARTICLE:5} on the Weil conjectures, Deligne extended the concept of weights to Galois representations and put weight filtrations on the $l$-adic cohomology groups of varieties over finite fields. In analogy to the existence of a limiting mixed Hodge structure in Hodge theory, Deligne's theory of weights is expected to extend further to local fields, over which we refer to mixed representations as mixed monodromy-weight structures. Mixed Hodge, twistor and monodromy-weight structures share many features, for example they form abelian categories. \\
Weights are closely related to slopes: Both control when nonzero morphisms are admissible. The key difference is that slope filtrations are descending in terms of the slopes of the graded pieces, while weight filtrations are ascending. Based on André's formalisation of slope filtrations \cite{ARTICLE:1}, we will set up an abstract framework within which the shared properties of mixed structures become a formal consequence of the fact that the weights of the graded pieces increase along a weight filtration. \\

The appropriate categories for this purpose are quasi-abelian categories, which possess many properties of abelian categories, but in which epi-monics need not be isomorphisms. In our setting, slopes will be the obstruction to isomorphy. \\
Let $\mathcal{A}$ be a quasi-abelian category with a slope function $\mu$ taking values in a totally ordered, uniquely divisible abelian group $\Lambda$. We recall these notions in \ref{sec:2}; they are adopted from \cite{ARTICLE:1}, with the difference that we require slopes to strictly increase along epi-monics which fail to be isomorphisms. Let $\text{Fun}(\Lambda,\mathcal{A})$ denote the category of functors $\Lambda \to \mathcal{A}$ and let $W\mathcal{A}^\text{pre} \subset \text{Fun}(\Lambda,\mathcal{A})$ be the full subcategory of finite ascending $\Lambda$-indexed filtrations, whose objects we indicate as pairs $(M = \text{colim}~W \in \mathcal{A},W)$. Our starting point is the following definition (\ref{def:1}): Let $W\mathcal{A} \subset W\mathcal{A}^\text{pre}$ be the full subcategory of pairs $(M,W)$ with the property that for all $\lambda \in \Lambda$, $\text{gr}^W_\lambda M$ is zero or semistable of slope $\lambda$. We call an object $(M,W)$ of $W\mathcal{A}$ a mixed structure with weight filtration $W$. \\
The basic observation is that when $(M,W), (M',W') \in W\mathcal{A}$ have a single break at different slopes, in which case we refer to them as pure structures, then there are no nozero morphisms $(M,W) \to (M',W')$ in $W\mathcal{A}$. This is enforced by the weight filtrations being ascending and it leads to stronger exactness properties than one generally has for slope filtrations. \\

While the motivating examples are instances of the category $W\mathcal{A}$, at the abstract level, we establish more refined results involving the full subcategory $W\mathcal{A}^\text{pre}_{\text{uds}} \subset W\mathcal{A}^\text{pre}$ of pairs $(M,W)$ with the property that for all $\lambda \in \Lambda$, all nonzero subobjects of $\text{gr}^W_\lambda M$ have slope less than or equal to $\lambda$, and its dual version $W\mathcal{A}^\text{pre}_\text{udq}$ (\ref{def:5}). In \ref{subsec:3.1}, we study the class of morphisms in $W\mathcal{A}^\text{pre}$ whose domain belongs to $W\mathcal{A}^\text{pre}_\text{udq}$ and whose codomain is in $W\mathcal{A}^\text{pre}_\text{uds}$. \\
The main results are \ref{cor:1} and \ref{cor:2}. They assert:

\begin{thmA}
\namedlabel{thm:A}{Theorem A}
Let $f: M \to N$ be a morphism in $W\mathcal{A}^\text{pre}$ with $M \in W\mathcal{A}^\text{pre}_\text{udq}$ and $N \in W\mathcal{A}^\text{pre}_\text{uds}$.
\begin{itemize}
\item[(1)] $f$ is strict in $W\mathcal{A}^\text{pre}$, that is, the canonical morphism $\text{Coim}f \to \text{Im}f$ in $W\mathcal{A}^\text{pre}$ is an isomorphism.
\item[(2)] The cokernel of $f$ in $W\mathcal{A}^\text{pre}$ is computed in $\text{Fun}(\Lambda,\mathcal{A})$ and it lies in $W\mathcal{A}^\text{pre}_\text{uds}$. Dually, $\text{Ker}f \in W\mathcal{A}^\text{pre}_\text{udq}$.
\end{itemize}
\end{thmA}

The key technical step on the way to these results is the demonstration that any such $f$ is strictly compatible and strictly co-compatible with the filtrations, which we do in \ref{prop:2}. Strict compatibility and strict co-compatibility with the filtrations, which are weaker notions than strictness in $W\mathcal{A}^\text{pre}$, are the content of \ref{def:3}. \\
\ref{thm:A} specialises to statements about the category $W\mathcal{A}$ (\ref{thm:1}):

\begin{thmB}
\namedlabel{thm:B}{Theorem B}
$\text{}$
\begin{itemize}
\item[(1)] The category $W\mathcal{A}$ is abelian.
\item[(2)] The inclusion $W\mathcal{A} \hookrightarrow \text{Fun}(\Lambda,\mathcal{A})$ is (strongly) exact.
\end{itemize}
\end{thmB}

As we demonstrate in \ref{sec:4}, \ref{thm:B} comprises the respective known results for mixed Hodge structures \cite[Théorème 1.2.10, Théorème 2.3.5]{ARTICLE:2}, mixed twistor structures \cite[Lemma 1.3]{MISC:3} and mixed monodromy-weight structures (see \ref{rmk:3}(2) for a brief survey of the literature).  \\

When $\mathcal{A}$ is an arbitrary quasi-abelian category, it is necessary to distinguish between strict compatibility and strict co-compatibility of morphisms in $W\mathcal{A}^\text{pre}$ with the filtrations. In the appendix, we analyse under which condition on $\mathcal{A}$ this distinction becomes obsolete. The same condition, for which we borrow from lattice theory the name of modular law, turns out to determine whether the category $W\mathcal{A}^\text{pre}$ is again quasi-abelian. \\
Further material in the appendix includes a discussion of the case that $\mathcal{A}$ comes equipped with a tensor product.

\paragraph{Acknowledgements} My thanks go to my advisor Moritz Kerz. The idea for the project is due to him. I greatly appreciate his continuous feedback, which has led to simplifications of many of my original proofs, and his suggestions for improving the presentation. \\
I thank Tony Scholl for providing me with a copy of his notes for his talk \cite{MISC:5}. \\
The author was supported by the CRC 1085 \enquote{Higher Invariants} funded by the Deutsche Forschungsgemeinschaft (DFG).

\section{Preliminaries}
\namedlabel{sec:2}{Section \thesection}

\begin{notation}
We let $\mathcal{A}^\circ$ denote the opposite category of a category $\mathcal{A}$. The same notation is used on objects, morphisms and functors.
\end{notation}

\begin{intermediate}
For our purposes, assuming categories to be essentially small will not be a restriction. This assumption will be met in all examples of interest.
\end{intermediate}

\begin{notation}
Let $\mathcal{A}$ be an additive category with kernels and cokernels. We write $\text{ker}f$ for the kernel of a morphism $f$ in $\mathcal{A}$ and $\text{Ker}f$ for the domain of $\text{ker}f$. Analogous notation is used for cokernels. A sequence $0 \to M \stackrel{f}{\to} N \stackrel{g}{\to} P \to 0$ in $\mathcal{A}$ is said to be a \textit{short exact sequence} if $f = \text{ker}g$ and $g = \text{coker}f$. \\
Every morphism $f: M \to N$ in $\mathcal{A}$ admits a canonical factorisation $M \to \text{Coim}f \to \text{Im}f \to N$, where $\text{Coim}f = \text{Coker}\!\left(\text{ker}f\right)$ and $\text{Im}f = \text{Ker}\!\left(\text{coker}f\right)$. We say that $f$ is \textit{strict} if $\text{Coim}f \to \text{Im}f$ is an isomorphism. The kernels are precisely the strict morphisms which are monic, the \textit{strict monics} for short, and the cokernels are precisely the \textit{strict epis}. We will refer to monics and epis as \textit{subobjects} and \textit{quotients} respectively when there is particular emphasis on the codomain respectively domain. \\
With that, for strict subobjects $i_A: A \hookrightarrow X$ and $i_B: B \hookrightarrow X$ in $\mathcal{A}$, we put $X/A \defeq \text{Coker}\!\left(i_A\right)$ and $A \cap B \defeq \text{Ker}\!\left(X \to X/A \oplus X/B\right)$. Then $A \cap B$ is a pullback of the cospan formed by $i_A$ and $i_B$, so $A \cap B \cong \text{Ker}\!\left(\left(i_A,-i_B\right): A \oplus B \to X\right) \cong \text{Ker}\!\left(\left(i_A,i_B\right)\right)$. We set $A + B \defeq \text{Im}\!\left(\left(i_A,i_B\right)\right)$. For a morphism $f: M \to N$ in $\mathcal{A}$ and strict subobjects $P \hookrightarrow M$ and $Q \hookrightarrow N$, we write $f(P) \defeq \text{Im}(P \hookrightarrow M \stackrel{f}{\to} N)$ and let $f^{-1}(Q)$ denote the pullback of $Q \hookrightarrow N$ along $f$.
\end{notation}

\begin{definition}
A \textit{quasi-abelian category} \cite[Section 1.1]{ARTICLE:6}, \cite[Section 1.2.7]{ARTICLE:1} is an additive category with kernels and cokernels such that kernels are stable under pushout and cokernels are stable under pullback. \\
Equivalently, $\mathcal{A}$ is an additive category which occurs as part of a cotilting torsion pair $\left(\mathcal{TL}_\mathcal{A},\mathcal{A}\right)$ in an abelian category $\mathcal{L}_\mathcal{A}$, the \textit{left abelian envelope} of $\mathcal{A}$ \cite[Section 1.2]{ARTICLE:6}, \cite[Section 5.4, Proposition B.3]{MISC:2}. $\mathcal{L}_\mathcal{A}$ and $\mathcal{TL}_\mathcal{A}$ are determined by $\mathcal{A}$. Likewise, $\mathcal{A}$ is part of a tilting torsion pair $\left(\mathcal{A},\mathcal{TR}_\mathcal{A}\right)$ in an abelian category $\mathcal{R}_\mathcal{A}$, the \textit{right abelian envelope} of $\mathcal{A}$. Also $\mathcal{A}^\circ$ is quasi-abelian, and one has $\mathcal{R}_\mathcal{A} \simeq \left(\mathcal{L}_{\mathcal{A}^\circ}\right)^\circ$ as well as $\mathcal{TR}_\mathcal{A} \simeq \left(\mathcal{TL}_{\mathcal{A}^\circ}\right)^\circ$.
\end{definition}

\begin{intermediate} 
These axioms make for a certain amount of homological algebra in $\mathcal{A}$, their upshot being that $\text{Ext}^1$ is bifunctorial on $\mathcal{A}$. Indeed, one can regard $\mathcal{A}$ as an exact category with all kernels and cokernels and equipped with the maximal exact structure, that is the admissible monics are the kernels and the admissible epis are the cokernels.
\end{intermediate}

\begin{remark}
\namedlabel{rmk:6}{Remark \thesubsection}
Let $\mathcal{A}$ be a quasi-abelian category. We will use the following properties (and may do so without repeated reference).
\begin{itemize}
\item[(1)] If $A \hookrightarrow B$ and $B \hookrightarrow C$ are strict subobjects in $\mathcal{A}$, then so is their composite $A \hookrightarrow C$ and the sequence
\begin{equation*}
0 \to \frac{B}{A} \to \frac{C}{A} \to \frac{C}{B} \to 0
\end{equation*}
is exact \cite[Proposition 1.1.7]{ARTICLE:6}, \cite[Lemma 1.2.5(2)]{ARTICLE:1}. \\
On the other hand, it follows immediately from the axioms of a quasi-abelian category that for a strict subobject $X \hookrightarrow C/A$ and $Y \defeq (C \twoheadrightarrow C/A)^{-1}(X)$, the resulting sequence $0 \to A \to Y \to X \to 0$ is exact. \\
Thus one can speak of strict subquotients without ambiguity, and moreover of their intersection and sum: For any other strict subobject $B' \hookrightarrow C$ such that $A \hookrightarrow C$ factors through $B' \hookrightarrow C$, it holds that
\begin{equation*}
\frac{B \cap B'}{A} \hookrightarrow \frac{B}{A} \cap \frac{B'}{A}, \quad \frac{B + B'}{A} \hookrightarrow \frac{B}{A} + \frac{B'}{A}
\end{equation*}
are isomorphisms of strict subobjects of $C/A$.
\item[(2)] For any morphism $f$ in $\mathcal{A}$, the canonical map $\text{Coim}f \to \text{Im}f$ is epi-monic \cite[Corollary 1.1.5]{ARTICLE:6}, \cite[Lemma 1.2.5(4)]{ARTICLE:1}. \\
Recall that as an additive category with kernels and cokernels, $\mathcal{A}$ is abelian if and only if this map is an isomorphism for all $f$, in other words if and only if all morphisms in $\mathcal{A}$ are strict.
\item[(3)] If $f$ and $g$ are composable morphisms in $\mathcal{A}$ such that $gf$ is strict monic, then so is $f$ (the first morphism in the composition). Dually, if a composite morphism is strict epi, so is the second morphism in the composition \cite[Proposition 1.1.8]{ARTICLE:6}.
\item[(4)] $\mathcal{A}$ is closed under subobjects in its left abelian envelope $\mathcal{L}_\mathcal{A}$ \cite[5.4]{MISC:2}, \cite[Proposition 1.2.14]{ARTICLE:1} and, consequently, a short sequence in $\mathcal{A}$ is exact in $\mathcal{A}$ if and only if it is exact in $\mathcal{L}_\mathcal{A}$. The resulting morphism $K_0(\mathcal{A}) \to K_0\!\left(\mathcal{L}_\mathcal{A}\right)$ is an isomorphism. Here, $K_0(\mathcal{A})$ denotes the Grothendieck group of $\mathcal{A}$, built from the free group on the set of isomorphism classes $\text{sk}(\mathcal{A})$ of $\mathcal{A}$ by taking short exact sequences as relations.  \\
Dually, $\mathcal{A}$ is closed under quotients in its right abelian envelope $\mathcal{R}_\mathcal{A}$ and a short sequence in $\mathcal{A}$ is exact in $\mathcal{A}$ if and only if it is exact in $\mathcal{R}_\mathcal{A}$. \\
This implies that kernels of all morphisms and cokernels of strict morphisms in $\mathcal{A}$ are computed in $\mathcal{L}_\mathcal{A}$, while kernels of strict morphisms and cokernels of all morphisms in $\mathcal{A}$ are computed in $\mathcal{R}_\mathcal{A}$. \\
Furthermore, $\mathcal{A}$ is closed under extensions both in $\mathcal{L}_\mathcal{A}$ and in $\mathcal{R}_\mathcal{A}$.
\end{itemize}
\end{remark}

\begin{definition}
\namedlabel{def:6}{Definition \thesubsection}
Let $\mathcal{A}$ be a quasi-abelian category and $\Lambda$ a totally ordered, uniquely divisible abelian group. We mainly follow \cite[Sections 1.2.6, 1.3-1.4.2]{ARTICLE:1}.
\begin{itemize}
\item[(1)] An \textit{ascending $\Lambda$-indexed filtration on $M \in \mathcal{A}$} is a functor $W_{\leq(-)}M: \Lambda \to \mathcal{A}$ with the following properties: (Exhaustiveness:) It admits a colimit cone with vertex $M$ such that $W_{\leq\lambda}M \to M$ is a strict subobject for all $\lambda \in \Lambda$. (Separatedness and right continuity:) The following limits exist and are given as indicated: $\lim W_{\leq(-)}M = 0$ and for all $\lambda \in \Lambda$, $\lim_{\lambda'>\lambda} W_{\leq\lambda'}M = W_{\leq\lambda}M$, with the obvious cone maps. \\
Define a \textit{descending $\Lambda$-indexed filtration on $M$}, $F^{\geq(-)}M: \Lambda^\circ \to \mathcal{A}$, to be an ascending $\Lambda^\circ$-indexed filtration on $M$. Thereby, right continuity gets replaced by left continuity, that is one requires $\lim_{\lambda'<\lambda} F^{\geq\lambda'}M = F^{\geq\lambda}M$ for all $\lambda \in \Lambda$.
\item[(2)] A \textit{rank function} on $\mathcal{A}$ is an assignment $\text{rk}: \text{sk}(\mathcal{A}) \to \mathbb{N}$ which is additive along short exact sequences, in other words extends to a group homomorphism $K_0(\mathcal{A}) \to \mathbb{Z}$, and such that $\text{rk}(M) \neq 0$ holds for $M \in \text{sk}(\mathcal{A})\backslash\{0\}$.
\end{itemize}
If $\mathcal{A}$ admits a rank function, then every flag, that is strictly ascending chain of strict subobjects, on an object of $\mathcal{A}$ is finite, and conversely, if the maximal length of a flag on any object is finite, it defines a rank function on $\mathcal{A}$ (apply \cite[Remark 1.2.10]{ARTICLE:1} to the left abelian envelope). \\

In what follows, assume $\mathcal{A}$ to be equipped with a rank function. In this situation, any ascending filtration $W_{\leq(-)}M$ on an object $M \in \mathcal{A}$ in the sense of item (1) is \textit{finite}, that is isomorphic to a finite flag $0 \hookrightarrow W_{\leq\lambda_0}M \hookrightarrow W_{\leq\lambda_1}M \hookrightarrow ... \hookrightarrow W_{\leq\lambda_r}M = M$ with $\lambda_0 < \lambda_1 < ... < \lambda_r$, called the \textit{breaks} of the filtration, namely such that $W_{\leq(-)}M$ is isomorphic to a constant functor precisely on the intervals $\lambda_{i-1} \leq \lambda < \lambda_i$. It is then clear that the colimit $W_{<\lambda}M \defeq \text{colim}_{\lambda'<\lambda} W_{\leq\lambda'}M$ exists for all $\lambda$ and that the canonical morphism $W_{<\lambda}M \to W_{\leq\lambda}M$ is strict monic. We let $\text{gr}^W_\lambda M$ denote its cokernel. \\
Analogously, we have $F^{>\lambda}M \defeq \text{colim}_{\lambda'>\lambda} F^{\geq\lambda'}M$ and $\text{gr}^\lambda_F M \defeq \text{Coker}\!\left(F^{>\lambda}M \hookrightarrow F^{\geq\lambda}M\right)$ for a descending filtration $F^{\geq(-)}M$ on $M$.
\begin{itemize}
\item[(3)] A \textit{degree function} on $\mathcal{A}$ is an assignment $\text{deg}: \text{sk}(\mathcal{A}) \to \Lambda$ which is additive along short exact sequences, equivalently factors through a homomorphism $K_0(\mathcal{A}) \to \Lambda$. We call $\mu \defeq \text{deg}/\text{rk}: \text{sk}(\mathcal{A})\backslash\{0\} \to \Lambda$ a \textit{$\Lambda$-valued slope function} if it satisfies:
\begin{equation}
\setlength{\jot}{0pt}
\label{eq:dagger}
\tag{$\dagger$}
\begin{split}
&\text{Whenever there is an epi-monic } \varphi: M \to N \text{ in } \mathcal{A}, \text{ it holds that } (\dagger')~\mu(M) \leq \mu(N), \\
&\text{with equality if and only if } \varphi \text{ is an isomorphism.}
\end{split}
\end{equation}
\end{itemize}
For our purposes, we impose a stronger condition than in \cite[Definition 1.3.1]{ARTICLE:1}, where only $(\dagger')$ is required. It is strictly stronger: While $\mu = 0$ is a slope function in the sense of \cite{ARTICLE:1}, it is a slope function in our sense only if $\mathcal{A}$ is abelian. \\
It follows from \cite[Lemma 1.2.17]{ARTICLE:1} that in the case the rank of the torsion classes, that is the image of $\text{sk}\!\left(\mathcal{TL}_\mathcal{A}\right)$ under $\text{rk}: K_0\!\left(\mathcal{L}_\mathcal{A}\right) \cong K_0(\mathcal{A}) \to \mathbb{Z}$, is zero, condition \eqref{eq:dagger} means that $\text{deg}: K_0\!\left(\mathcal{L}_\mathcal{A}\right) \cong K_0(\mathcal{A}) \to \Lambda$ is positive on $\text{sk}\!\left(\mathcal{TL}_\mathcal{A}\right)\backslash\{0\}$, see also \cite[Corollary 1.4.10]{ARTICLE:1}.
\begin{itemize}
\item[(4)] Suppose $\mu$ is a slope function on $\mathcal{A}$. A nonzero object $M \in \mathcal{A}$ is called \textit{semistable} with respect to $\mu$ if for all nonzero (strict) subobjects $N \to M$, it holds that $\mu(N) \leq \mu(M)$. Equivalently \cite[Lemma 1.3.7(2)]{ARTICLE:1}, $M$ is semistable if it admits no quotients of strictly smaller slope. \\
For all $\lambda \in \Lambda$, the full subcategory $\mathcal{A}(\lambda) \subset \mathcal{A}$ consisting of zero and the semistable objects of slope $\lambda$ is an abelian category. Indeed, it is closed under biproducts, kernels and cokernels in $\mathcal{A}$ \cite[Lemma 1.3.9(1)]{ARTICLE:1} and, consequently, every epi-monic in $\mathcal{A}(\lambda)$ is epi-monic in $\mathcal{A}$, hence an isomorphism due to \eqref{eq:dagger}. \\
Any nonzero object $M \in \mathcal{A}$ admits a \textit{universal destabilising subobject} $U \hookrightarrow M$ which is characterised by the properties that $U$ is semistable and that for any nonzero strict subobject $N \hookrightarrow M$, either $\mu(N) < \mu(U)$ or $N \hookrightarrow M$ factors through $U \hookrightarrow M$. Observe that in any case, it holds that $\mu(N) \leq \mu(U)$. \\
The \textit{slope filtration} corresponding to the slope function $\mu$ under \cite[Theorem 1.4.7]{ARTICLE:1} is $F_\mu^{\geq(-)}: \Lambda^\circ \times \mathcal{A} \to \mathcal{A}$ such that for all $M \in \mathcal{A}$, $F_\mu^{\geq(-)}M$ is the unique (finite) descending $\Lambda$-indexed filtration on $M$ with the property that $\text{gr}^\lambda_{F_\mu}M \in \mathcal{A}(\lambda)$ holds for all $\lambda \in \Lambda$. Explicitly, if $U_0 \hookrightarrow M$ is the universal destabilising subobject, then the highest break is $\lambda_0 \defeq \mu\!\left(U_0\right)$ and $F_\mu^{\geq\lambda_0}M = U_0$, $\lambda_1$ is the slope of the universal destabilising subobject $U_1$ of $M/U_0$ and $F_\mu^{\geq\lambda_1}M = \left(M \to M/U_0\right)^{-1}\!\left(U_1\right)$, and so on.
\end{itemize}
\end{definition}

\begin{intermediate}
The non-example of a slope function in our sense, $\mu = 0$, shows that condition \eqref{eq:dagger} is necessary to ensure that morphisms between semistable objects of the same slope are strict. We will need this condition in \ref{sec:3}. In fact, \ref{thm:1}(1) does not extend to our non-example, for trivial reasons. \\

Fix a quasi-abelian category $\mathcal{A}$ and a totally ordered, uniquely divisible abelian group $\Lambda$. We now introduce the main objects of study.
\end{intermediate}

\begin{definition}
\namedlabel{def:1}{Definition \thesubsection}
$\text{}$
\begin{itemize}
\item[(1)] Let $\text{Fun}(\Lambda,\mathcal{A})$ denote the category of functors $\Lambda \to \mathcal{A}$, write $W\mathcal{A}^\text{pre,all} \subset \text{Fun}(\Lambda,\mathcal{A})$ for the full subcategory of ascending $\Lambda$-indexed filtrations (\ref{def:6}(1)) and further let $W\mathcal{A}^\text{pre} \subset W\mathcal{A}^\text{pre,all}$ be the full subcategory of finite filtrations. We indicate objects of $W\mathcal{A}^\text{pre,all}$ and $W\mathcal{A}^\text{pre}$ as pairs $(M,W)$, where $W$ is a filtration on $M \in \mathcal{A}$.
\end{itemize}
As remarked before, when $\mathcal{A}$ admits a rank function, there is no difference between $W\mathcal{A}^\text{pre}$ and $W\mathcal{A}^\text{pre,all}$. \\
From here on (and up to the appendix), assume in addition that $\mathcal{A}$ is equipped with a rank function and a $\Lambda$-valued slope function $\mu$. Semistability will always be with respect to $\mu$.
\begin{itemize}
\item[(2)] Let $W\mathcal{A} \subset W\mathcal{A}^\text{pre}$ be the full subcategory of objects $(M,W)$ with the property that $\text{gr}^W_\lambda M \in \mathcal{A}(\lambda)$ (\ref{def:6}(4)) holds for all $\lambda \in \Lambda$. Call objects $(M,W)$ of $W\mathcal{A}$ \textit{mixed structures} with \textit{weight filtration} $W$.
\end{itemize}
\end{definition}

\begin{intermediate}
If $f$ is a morphism in $W\mathcal{A}^\text{pre}$, we refer to the morphism induced by $f$ on the colimits of the filtrations as the \textit{underlying morphism} of $f$ in $\mathcal{A}$. When there is no risk of confusion between a morphism in $W\mathcal{A}^\text{pre}$ and its underlying morphism in $\mathcal{A}$, we may omit the filtrations from the notation. \\

Along with $\mathcal{A}$, the category $\text{Fun}(\Lambda,\mathcal{A})$ is quasi-abelian. The category $W\mathcal{A}^\text{pre}$ is at least additive and has all kernels and cokernels, which we check in \ref{prop:1} of the appendix. There, we also show that $W\mathcal{A}^\text{pre}$ is a quasi-abelian category precisely when the modular law (\ref{def:2}), a feature of abelian categories, holds in $\mathcal{A}$. In contrast, the category $W\mathcal{A}^\text{pre,all}$ is only additive in general. \\
One of the main results of the subsequent section is that the category $W\mathcal{A}$ is always abelian (\ref{thm:1}).
\end{intermediate}

\begin{example}
\namedlabel{ex:1}{Example \thesubsection}
For illustration, we anticipate \ref{ex:5}. Let $\mathcal{B}$ be an abelian category with a rank function $\text{rk}_\mathcal{B}$ and let $\mathcal{B}_2$ denote the quasi-abelian category of triples $(M,F,\bar{F})$, where $F$ and $\bar{F}$ are descending $\Lambda$-indexed filtrations on $M \in \mathcal{B}$, and whose morphisms are the morphisms in $\mathcal{B}$ respecting the filtrations. For $0 \neq \left(M,F,\bar{F}\right) \in \mathcal{B}_2$, put $\text{rk}\!\left(M,F,\bar{F}\right) \defeq \text{rk}_{\mathcal{B}}(M)$ and $\text{deg}\!\left(M,F,\bar{F}\right) \defeq \sum_{\bar{\nu},\nu \in \Lambda} \text{rk}_{\mathcal{B}}\!\left(\text{gr}^{\bar{\nu}}_{\bar{F}}\text{gr}^\nu_FM\right)\left(\bar{\nu}+\nu\right)$. Then deg/rk is a slope function on $\mathcal{B}_2$ such that for $\lambda \in \Lambda$, $\mathcal{B}_2(\lambda)$ consists of those objects $\left(M,F,\bar{F}\right) \in \mathcal{B}_2$ for which $\text{gr}^{\bar{\nu}}_{\bar{F}}\text{gr}^\nu_FM = 0$ holds whenever $\bar{\nu} + \nu \neq \lambda$. The category $W\mathcal{B}_2$ is the category of abstract mixed Hodge structures (\ref{def:4}) as considered (though not under this name) in \cite[Théorème 1.2.10]{ARTICLE:2}.
\end{example}

\begin{intermediate}
The guiding idea is the following.
\end{intermediate}

\begin{remark}
The semistable objects of $\mathcal{A}$ are precisely those objects of $\mathcal{A}$ whose slope filtration has a single break, so as the slope filtration is functorial and descending, the slope cannot strictly decrease along nonzero morphisms of semistable objects (compare \cite[Lemma 1.3.8]{ARTICLE:1}). However already in \ref{ex:1}, there are nonzero morphisms between semistable objects along which the slope strictly increases (\ref{rmk:4}(1), \ref{rmk:2}(1)), which indicates a failure of exactness of the slope filtration (\cite[Theorem 1.5.9, (1) $\implies$ (4)]{ARTICLE:1}). \\
In contrast, if we have $(M,W), \left(M',W'\right) \in W\mathcal{A}$ with a single break at different slopes, call them \textit{pure structures}, then $\text{Hom}_{W\mathcal{A}}\!\left((M,W),\left(M',W'\right)\right) = \{0\}$ is enforced by the fact that the weight filtrations are ascending. We will see in \ref{thm:1} that this in turn leads to strong exactness properties of the weight filtration.
\end{remark}

\section{Exactness properties of weight filtrations}
\namedlabel{sec:3}{Section \thesection}

\ref{subsec:3.1} collects the main results of the abstract theory. The application of interest is \ref{thm:1}.

\subsection{Formal aspects}
\namedlabel{subsec:3.1}{Section \thesubsection}

We begin with some preparations, keeping the setup and notation from the previous section. \\

If $\mathcal{A}$ is a quasi-abelian category with a slope function $\mu$, then $-\mu$ is a slope function on the opposite category $\mathcal{A}^\circ$. Considering $\mathcal{A}^\circ$ equipped with the slope function $-\mu$, we have the following duality result.

\begin{sublemma}
\namedlabel{lemma:1}{Lemma \thesubsubsection}
The obvious extension of the assignment $(M,W)^\circ \mapsto \left(M^\circ,\left(\left(M/W_{<-\lambda}M\right)^\circ\right)_\lambda\right)$ to a functor $o_\mathcal{A}: \left(W\mathcal{A}^\text{pre}\right)^\circ \to W\!\left(\mathcal{A}^\circ\right)^\text{pre}$ is an equivalence of categories which restricts to an equivalence between strictly full subcategories $\left(W\mathcal{A}\right)^\circ \simeq W\!\left(\mathcal{A}^\circ\right)$.
\end{sublemma}

\begin{proof}
The inverse equivalence is $\left(o_{\mathcal{A}^\circ}\right)^\circ$. As $\text{gr}^W_\lambda\!\left(o_\mathcal{A}\!\left((M,W)^\circ\right)\right) = \left(\text{gr}^W_{-\lambda}M\right)^\circ$, $o_\mathcal{A}$ and $\left(o_{\mathcal{A}^\circ}\right)^\circ$ restrict to an equivalence $\left(W\mathcal{A}\right)^\circ \simeq W\!\left(\mathcal{A}^\circ\right)$.
\end{proof}

\begin{subdefinition}
\namedlabel{def:3}{Definition \thesubsubsection}
A morphism $f: M \to N$ in $W\mathcal{A}^\text{pre}$ is \textit{strictly compatible with the filtrations} \cite[Definition 1.5.1]{ARTICLE:1} if the morphism of filtered objects
\begin{equation*}
\left(\text{Im}_\mathcal{A}f,\left(f\!\left(W_{\leq\lambda}M\right)\right)_\lambda\right) \to \left(\text{Im}_\mathcal{A}f,\left(f(M) \cap W_{\leq\lambda}N\right)_\lambda\right) = \text{Im}f
\end{equation*}
is an isomorphism. We say that $f$ is \textit{stricly co-compatible with the filtrations} if the morphism
\begin{equation*}
\text{Coim}f = \left(\text{Coim}_\mathcal{A}f,\left(\left(W_{\leq\lambda}M + \text{Ker}_\mathcal{A}f\right)/\text{Ker}_\mathcal{A}f\right)_\lambda\right) \to \left(\text{Coim}_\mathcal{A}f,\left(f^{-1}\!\left(W_{\leq\lambda}N\right)/\text{Ker}_\mathcal{A}f\right)_\lambda\right)
\end{equation*}
in $W\mathcal{A}^\text{pre}$ is an isomorphism. This amounts to requiring that the image $o_\mathcal{A}\!\left(f^\circ\right)$ of the opposite of $f$ under the equivalence $o_\mathcal{A}: \left(W\mathcal{A}^\text{pre}\right)^\circ \to W\!\left(\mathcal{A}^\circ\right)^\text{pre}$ from \ref{lemma:1} be strictly compatible with the filtrations in $W\!\left(\mathcal{A}^\circ\right)^\text{pre}$.
\end{subdefinition}

\begin{intermediate}
Note that a strict morphism in $W\mathcal{A}^\text{pre}$ is both strictly compatible and strictly co-compatible with the filtrations, but that the converse does not hold unless $\mathcal{A}$ is abelian. \\

In the appendix, we will see that strict compatibility and strict co-compatibility are equivalent for morphisms in $W\mathcal{A}^\text{pre}$ whose underlying morphism in $\mathcal{A}$ is strict (\ref{rmk:1}), and discuss under which hypothesis on the quasi-abelian category $\mathcal{A}$ they are equivalent for all morphisms in $W\mathcal{A}^\text{pre}$ (\ref{prop:4}). \\

Here, for the purpose of \ref{def:5}, we introduce the dual notion of the universal destabilising subobject, which we refer to as the universal destabilising quotient. We then prove a stability property needed in the sequel. The discussion of further properties is deferred to \ref{lemma:5}.
\end{intermediate}

\begin{subremark}
For $X \in \mathcal{A}$, let $(X/V)^\circ \hookrightarrow X^\circ$ be the universal destabilising subobject of $X^\circ \in \mathcal{A}^\circ$ with respect to the slope function $-\mu$ on $\mathcal{A}^\circ$ and say $(-\mu)\!\left((X/V)^\circ\right) = -\lambda$. Then $X \twoheadrightarrow X/V$ in $\mathcal{A}$ is the semistable quotient of $X$ of minimal slope $\lambda$, that is every quotient of $X$ has slope $\geq \lambda$ and if $X/V'$ has slope equal to $\lambda$, $X \twoheadrightarrow X/V'$ factors through $X \twoheadrightarrow X/V$. We call $X/V$ the \textit{universal destabilising quotient} of $X$. \\
Explicitly, let $\lambda \in \Lambda$ be maximal with the property that $F^{\geq\lambda}X = X$. Then the universal destabilising quotient of $X$ is $X/F^{>\lambda}X$ and it has slope $\lambda$. Indeed, the slope filtration on $X^\circ \in \mathcal{A}^\circ$ with respect to $-\mu$ is given by $F^{\geq \nu}\!\left(X^\circ\right) = \left(X/F^{>-\nu}X\right)^\circ$. From it, we obtain the unique chain of quotients of $X$ in $\mathcal{A}$ with semistable kernel of descending slope, terminating at the universal destabilising quotient.
\end{subremark}

\begin{sublemma}
\namedlabel{lemma:8}{Lemma \thesubsubsection}
Write $\text{UDS}(A) \hookrightarrow A$ for the universal destabilising subobject of $A\in \mathcal{A}$ and $A \twoheadrightarrow \text{UDQ}(A)$ for the universal destablising quotient. Let $0 \to X \to Y \stackrel{\pi}{\to} Z \to 0$ be a short exact sequence in $\mathcal{A}$ and suppose that $\mu\!\left(\text{UDS}(Y)\right) \leq \lambda$ and $\mu(X) = \lambda$ ($X$ is automatically semistable). Then $\mu\!\left(\text{UDS}(Z)\right) \leq \lambda$. \\
Dually, if $\mu\!\left(\text{UDQ}(Y)\right) \geq \lambda$ and $\mu(Z) = \lambda$, then $\mu\!\left(\text{UDQ}(X)\right) \geq \lambda$.
\end{sublemma}

\begin{proof}
For a strict subobject $U \hookrightarrow Z$, look at the resulting short exact sequence
\begin{equation*}
0 \to X \to \pi^{-1}(U) \to U \to 0.
\end{equation*}
As $\mu(X) = \lambda$ and $\mu\!\left(\pi^{-1}(U)\right) \leq \lambda$ by assumption, we must have $\mu(U) \leq \lambda$ according to \cite[Lemma 1.3.4(1)]{ARTICLE:1}. Explicitly, $\mu(U) = r^{-1}\!\left(\text{rk}\!\left(\pi^{-1}(U)\right)\mu\!\left(\pi^{-1}(U)\right) - \text{rk}(X)\mu(X)\right) \leq r^{-1}r\lambda = \lambda$ with $r \defeq \text{rk}\!\left(\pi^{-1}(U)\right) - \text{rk}(X)$.
\end{proof}

\begin{subdefinition}
\namedlabel{def:5}{Definition \thesubsubsection}
Let $W\mathcal{A}^\text{pre}_\text{uds}$ denote the full subcategory of $W\mathcal{A}^\text{pre}$ whose objects $(M,W)$ have the property that for all $\lambda \in \Lambda$, all nonzero subobjects of $\text{gr}^W_\lambda M$ have slope at most $\lambda$, equivalently, the universal destabilising subobject of $\text{gr}^W_\lambda M$ has slope at most $\lambda$ (in symbols, $\mu\!\left(\text{UDS}\!\left(\text{gr}^W_\lambda M\right)\right) \leq \lambda$). \\
Dually, let $W\mathcal{A}^\text{pre}_\text{udq}$ denote the full subcategory of $W\mathcal{A}^\text{pre}$ whose objects $(M,W)$ have the property that for all $\lambda$, all nonzero quotients of $\text{gr}^W_\lambda M$ have slope at least $\lambda$, equivalently, the universal destabilising quotient of $\text{gr}^W_\lambda M$ has slope at least $\lambda$ ($\mu\!\left(\text{UDQ}\!\left(\text{gr}^W_\lambda M\right)\right) \geq \lambda$).
\end{subdefinition}

\begin{intermediate}
$W\mathcal{A}^\text{pre}_\text{uds}$ is a strictly full subcategory of $W\mathcal{A}^\text{pre}$, closed under biproducts and kernels in $W\mathcal{A}^\text{pre}$. For $(M,W) \in W\mathcal{A}^\text{pre}_\text{uds}$ and every $\lambda$, the universal destabilising subobject of $W_{\leq\lambda}M$ has slope at most $\lambda$, see \ref{lemma:7} below. Dual statements apply to $W\mathcal{A}^\text{pre}_\text{udq}$, the second reading as follows: For $(M,W) \in W\mathcal{A}^\text{pre}_\text{udq}$ and every $\lambda$, the universal destabilising quotient of $M/W_{<\lambda}M$ has slope at least $\lambda$. \\
It holds that $\left(W\mathcal{A}^\text{pre}_\text{uds}\right)^\circ \simeq W\!\left(\mathcal{A}^\circ\right)^\text{pre}_\text{udq}$ under the equivalence from \ref{lemma:1}. Moreover, $W\mathcal{A} = W\mathcal{A}^\text{pre}_\text{uds} \cap W\mathcal{A}^\text{pre}_\text{udq}$ (inside $W\mathcal{A}^\text{pre}$) and thus $W\mathcal{A}$ is again an additive category.
\end{intermediate}

\begin{sublemma}
\namedlabel{lemma:7}{Lemma \thesubsubsection}
Let $N \in W\mathcal{A}^\text{pre}_\text{uds}$. For every $\lambda \in \Lambda$, the universal destabilising subobject of $W_{\leq\lambda}N$, whenever it is nonzero, has slope at most $\lambda$.
\end{sublemma}

\begin{proof}
The statement is true when the filtration $W$ of $N$ has no more than one break, and we proceed by induction on the number of breaks. Suppose $N \neq 0$ and let $\nu$ be the smallest break of $W$, then the quotient $N/W_{\leq\nu}N$ in $W\mathcal{A}^\text{pre}$ again belongs to $W\mathcal{A}^\text{pre}_\text{uds}$. If for some $\lambda \geq \nu$, the universal destabilising subobject $U$ of $W_{\leq\lambda}N$ had slope $> \lambda$, its image in $W_{\leq\lambda}N/W_{\leq\nu}N$ would be zero by induction. But then $U$ would be a (strict) subobject of $W_{\leq\nu}N$, hence zero or of slope $\leq \nu \leq \lambda$, a contradiction.
\end{proof}

\begin{intermediate}
Now we can formulate the key proposition.
\end{intermediate}

\begin{subproposition}
\namedlabel{prop:2}{Proposition \thesubsubsection}
If $f:M \to N$ is a morphism in $W\mathcal{A}^\text{pre}$ with $M \in W\mathcal{A}^\text{pre}_\text{udq}$ and $N \in W\mathcal{A}^\text{pre}_\text{uds}$, then $f$ is strictly compatible and strictly co-compatible with the filtrations.
\end{subproposition}

\begin{proof}
As the assumptions on $f$ are preserved under duality, it suffices to prove that $f$ is strictly compatible with the filtrations. \\
Call $(*_\lambda)$ the statement that $f\!\left(W_{\leq\lambda}M\right) \hookrightarrow f(M) \cap W_{\leq\lambda}N$ is an isomorphism. So we have to show that $(*_\lambda)$ holds for all $\lambda \in \Lambda$. The proof will proceed by induction on the number of breaks of the filtration of the domain $M$. At each step in the induction, letting $\mu$ denote the smallest break of the filtration of $M$, we will be able to deduce from the induction hypothesis that $(*_\lambda)$ holds for $\lambda \geq \mu$. To get it for all $\lambda$, we will need the preparatory statement \eqref{eq:+}, which we prove independently of the induction. Throughout, the case $f\!\left(W_{\leq\mu}M\right) = 0$ will be trivial, so we may assume that $f\!\left(W_{\leq\mu}M\right) \neq 0$. \\
First, for $\mu$ defined above, let us show that
\begin{equation}
\label{eq:+}
\tag{+}
f\!\left(W_{\leq\mu}M\right) \cap W_{<\mu}N = 0. 
\end{equation}
On the one hand, \ref{lemma:7} tells us that as a (strict) subobject of $W_{<\mu}N$, $f\!\left(W_{\leq\mu}M\right) \cap W_{<\mu}N$ has slope $< \mu$ unless it is zero. One the other hand, $f\!\left(W_{\leq\mu}M\right)$ admits an epi from $W_{\leq\mu}M$, hence is of slope $\geq \mu$, and $f\!\left(W_{\leq\mu}M\right)/\left(f\!\left(W_{\leq\mu}M\right) \cap W_{<\mu}N\right)$ admits a monic to $\text{gr}^W_\mu N$, hence is zero or of slope $\leq \mu$. So a slope estimate as in the proof of \ref{lemma:8} shows that $f\!\left(W_{\leq\mu}M\right) \cap W_{<\mu}N$ must have slope $\geq \mu$ if it is nonzero and we conclude that it is zero. This establishes \eqref{eq:+}. \\
Now we turn to verifying by induction on the number of breaks of the filtration of $M$ that $(*_\lambda)$ holds for all $\lambda$, that is $f$ is strictly compatible with the filtrations as claimed. As induction start, we can take either the case of no or of precisely one break, for which everything has already been shown. \\
We aim to apply the induction hypothesis to the morphism $\bar{f}: M/W_{\leq\mu}M \to X \defeq N/f\!\left(W_{\leq\mu}M\right)$ in $W\mathcal{A}^\text{pre}$. To this end, we need to show that
\begin{equation}
\label{eq:++}
\tag{++}
X \in W\mathcal{A}^\text{pre}_\text{uds}.
\end{equation}
According to \eqref{eq:+}, the composite $W_{\leq\mu}M \to f\!\left(W_{\leq\mu}M\right) \stackrel{h}{\to} \text{gr}^W_\mu N$ in $\mathcal{A}$ is an epi followed by a monic, which implies that $f\!\left(W_{\leq\mu}M\right)$ is (semistable) of slope $\mu$. So $\text{Coim}(h) = f\!\left(W_{\leq\mu}M\right)$ has slope $\mu$, while $\text{Im}(h)$ has slope $\leq \mu$, and it follows from condition \eqref{eq:dagger} on the slope function that $h$ is strict monic. The resulting isomorphism $f\!\left(W_{\leq\mu}M\right) \to \text{Im}(h)$ factors through the canonical epis $f\!\left(W_{\leq\mu}M\right) \to \left(W_{\leq\lambda}N + f\!\left(W_{\leq\mu}M\right)\right)/W_{\leq\lambda}N$ for all $\lambda < \mu$, which are thus themselves isomorphisms (\ref{rmk:6}(3)) and whose inverses give rise to split exact sequences $0 \to W_{\leq\lambda}N \to W_{\leq\lambda}N + f\!\left(W_{\leq\mu}M\right) \to f\!\left(W_{\leq\mu}M\right) \to 0$. Hence we see that $\text{gr}^W_\lambda N \to \text{gr}^W_\lambda X$ is an isomorphism whenever $\lambda \neq \mu$, whereas $\text{gr}^W_\mu X \cong \text{gr}^W_\mu N/f\!\left(W_{\leq\mu}M\right)$. From \ref{lemma:8}, we infer that every strict subobject of $\text{gr}^W_\mu X$ has slope at most $\mu$, which proves \eqref{eq:++}. \\
Therefore, $\bar{f}: M/W_{\leq\mu}M \to N/f\!\left(W_{\leq\mu}M\right)$ is strictly compatible with the filtrations by induction. For $\lambda \geq \mu$, this means that
\begin{equation*}
\frac{f\!\left(W_{\leq\lambda}M\right)}{f\!\left(W_{\leq\mu}M\right)} \hookrightarrow \frac{f(M) \cap W_{\leq\lambda}N}{f\!\left(W_{\leq\mu}M\right)} = \frac{f(M)}{f\!\left(W_{\leq\mu}M\right)} \cap \frac{W_{\leq\lambda}N}{f\!\left(W_{\leq\mu}M\right)}
\end{equation*}
is an isomorphism. It follows that $f\!\left(W_{\leq\lambda}M\right) \hookrightarrow f(M) \cap W_{\leq\lambda}N$ is an isomorphism when $\lambda \geq \mu$, that is $(*_\lambda)$ holds when $\lambda \geq \mu$. Thanks to \eqref{eq:+}, we get $(*_\lambda)$ also for $\lambda < \mu$, because $(*_\mu)$ together with \eqref{eq:+} implies that $f(M) \cap W_{<\mu}N = f(M) \cap W_{\leq\mu}N \cap W_{<\mu}N = f\!\left(W_{\leq\mu}M\right) \cap W_{<\mu}N = 0$. With this, the induction is complete.
\end{proof}

\begin{intermediate}
If desired, one could reduce the proof of the proposition to the case that $f$ is epi-monic. However, we did not see a gain in efficiency. \\

\ref{prop:2} is the main technical step in establishing the following stronger result.
\end{intermediate}

\begin{subcorollary}
\namedlabel{cor:1}{Corollary \thesubsubsection}
Any $f$ as in \ref{prop:2} is strict in $W\mathcal{A}^\text{pre}$ and $\text{Coim}f \cong \text{Im}f \in W\mathcal{A}$.
\end{subcorollary}

\begin{proof}
Strict compatibility and strict co-compatibility of $f$ with the filtrations together imply that $\text{gr}\!\left(\text{Coim}f \to \text{Im}f\right)$ is epi-monic, thus an isomorphism by the assumption that the domain of $f$ is in $W\mathcal{A}^\text{pre}_\text{udq}$ and that the codomain is in $W\mathcal{A}^\text{pre}_\text{uds}$. It follows that $f$ is strict in $W\mathcal{A}^\text{pre}$ and that $\text{Coim}f \cong \text{Im}f \in W\mathcal{A}$.
\end{proof}

\begin{intermediate}
The following construction is supplementary. While we will not use it in our proofs, it helps clarify the meaning of \ref{cor:1} (\ref{rmk:7}).
\end{intermediate}

\begin{subconstruction}
\namedlabel{cstr:1}{Construction \thesubsubsection}
For $N \in W\mathcal{A}^\text{pre}_\text{uds}$, there exists a largest strict subobject $N' \hookrightarrow N$ in $W\mathcal{A}^\text{pre}$ with the property that $N' \in W\mathcal{A}$, that is, any other strict subobject $Q \hookrightarrow N$ in $W\mathcal{A}^\text{pre}$ with $Q \in W\mathcal{A}$ factors through $N' \hookrightarrow N$. \\
Dually, for $M \in W\mathcal{A}^\text{pre}_\text{udq}$, there exists a smallest strict quotient $M \twoheadrightarrow M''$ in $W\mathcal{A}^\text{pre}$ with the property that $M'' \in W\mathcal{A}$.
\end{subconstruction}

\begin{proof}
Let $\Lambda_0 = \left\{\lambda_0 \leq \lambda_1 \leq ... \leq \lambda_n\right\} \subset \Lambda$ be the subset consisting of all breaks $\lambda$ of the filtration of $N$ for which there exists a strict subobject $Q \hookrightarrow N$ in $W\mathcal{A}^\text{pre}$ with $Q \in W\mathcal{A}$ and $\text{gr}^W_\lambda Q \neq 0$. Then the breaks of the filtration of any strict subobject $Q \hookrightarrow N$ in $W\mathcal{A}^\text{pre}$ with $Q \in W\mathcal{A}$ are contained in $\Lambda_0$. So we recursively construct $N' \in W\mathcal{A}^\text{pre}$ on the finitely many intervals $\lambda < \lambda_0$, $\lambda_{i-1} \leq \lambda < \lambda_i$, $1 \leq i \leq n$, and $\lambda_n \leq \lambda$. That is, for every $\lambda_i \in \Lambda_0$, we specify $W_{<\lambda_i}N' \hookrightarrow W_{\leq\lambda_i}N'$ assuming $W_{<\lambda_i}N' = W_{\leq\lambda_{i-1}}N'$ already given. \\
If $N$ has no nonzero strict subobject in $W\mathcal{A}^\text{pre}$ which is in $W\mathcal{A}$, we can take $N' = 0$. Otherwise, $\Lambda_0$ is not empty, say contains at least $\lambda_0$. Put $W_{\leq\lambda}N' \defeq 0$ for $\lambda < \lambda_0$, so that $W_{<\lambda_0}N' = 0$. For $\lambda \in \Lambda_0$, let $W_{\leq\lambda}N'$ be the pullback of the universal destabilising subobject $U_\lambda$ of $W_{\leq\lambda}N/W_{<\lambda}N'$ along $W_{\leq\lambda}N \to W_{\leq\lambda}N/W_{<\lambda}N'$. Then $N' \in W\mathcal{A}^\text{pre}$ so constructed verifies $\text{gr}^W_\lambda N' = 0$ whenever $\lambda \notin \Lambda_0$, while for $\lambda \in \Lambda_0$, we have $\text{gr}^W_\lambda N' = U_\lambda$, which is semistable. \\
We show by induction over $i$ that for all $\lambda_i \in \Lambda_0$, $U_{\lambda_i}$ has slope $\lambda_i$ and for any strict subobject $Q \hookrightarrow N$ with $Q \in W\mathcal{A}$, it holds that $W_{\leq\lambda_i}Q \hookrightarrow W_{\leq\lambda_i}N$ factors through $W_{\leq\lambda_i}N' \hookrightarrow W_{\leq\lambda_i}N$. By definition of $\Lambda_0$, it then follows that $Q \hookrightarrow N$ factors through $N' \to N$. \\
On the one hand, $U_{\lambda_0} \hookrightarrow W_{\leq\lambda_0}N$ must have slope $\geq \lambda_0$ since by choice of $\lambda_0$, $W_{\leq\lambda_0}N$ admits a subobject of slope $\lambda_0$. On the other hand, \ref{lemma:7} says that $U_{\lambda_0}$ has slope $\leq \lambda_0$. So $U_{\lambda_0}$ has slope equal to $\lambda_0$ and in turn, $W_{\leq\lambda_0}Q \hookrightarrow W_{\leq\lambda_0}N$ factors through $U_{\lambda_0} \hookrightarrow W_{\leq\lambda_0}N$. \\
In the induction step, let $i \geq 1$ be such that the claim has been established for the predecessor $\lambda_{i-1}$ of $\lambda_i$ in $\Lambda_0$ (while $i \leq n$). Suppose by contradiction that $U_{\lambda_i}$ has slope $> \lambda_i$. It follows that $W_{\leq\lambda_i}N'/\!\left(W_{\leq\lambda_i}N' \cap W_{<\lambda_i}N\right) = 0$, because in the composite of canonical maps $U_{\lambda_i} = W_{\leq\lambda_i}N'/W_{\leq\lambda_{i-1}}N' \to W_{\leq\lambda_i}N'/\!\left(W_{\leq\lambda_i}N' \cap W_{<\lambda_i}N\right) \to \text{gr}^W_{\lambda_i}N$, the first arrow is strict epi and the second arrow is monic. So $W_{\leq\lambda_i}N' \hookrightarrow W_{<\lambda_i}N$. Iterating this argument, we find that $W_{\leq\lambda_i}N'/\!\left(W_{\leq\lambda_i}N' \cap W_{\leq\lambda_{i-1}}N\right) = 0$. As then $W_{\leq\lambda_i}N'/W_{<\lambda_{i-1}}N' \hookrightarrow W_{\leq\lambda_{i-1}}N/W_{<\lambda_{i-1}}N'$, $W_{\leq\lambda_i}N'/W_{<\lambda_{i-1}}N'$ must have slope $\leq \lambda_{i-1}$ by the induction hypothesis. However, a slope estimate along the short exact sequence
\[\begin{tikzcd}[column sep=small]
0 & U_{\lambda_{i-1}} & {\frac{W_{\leq\lambda_i}N'}{W_{<\lambda_{i-1}}N'}} & {U_{\lambda_i}} & 0
\arrow[from=1-1, to=1-2]
\arrow[from=1-2, to=1-3]
\arrow[from=1-3, to=1-4]
\arrow[from=1-4, to=1-5]
\end{tikzcd}\]
shows that $W_{\leq\lambda_i}N'/W_{<\lambda_{i-1}}N'$ has slope $> \lambda_{i-1}$, a contradiction. \\
Therefore, $U_{\lambda_i}$ is zero or has slope $\leq \lambda_i$. Consequently, the morphism $\text{gr}^W_{\lambda_i}Q \to W_{\leq\lambda_i}N/W_{\leq\lambda_{i-1}}N'$ (which exists by induction) factors through $U_{\lambda_i} \hookrightarrow W_{\leq\lambda_i}N/W_{\leq\lambda_{i-1}}N'$, and by assumption, there is such $Q$ with $\text{gr}^W_{\lambda_i}Q \neq 0$. So in fact, $U_{\lambda_i}$ has slope $\lambda_i$ and $W_{\leq\lambda_i}Q \hookrightarrow W_{\leq\lambda_i}N$ lands in $W_{\leq\lambda_i}N'$. This finishes the induction. \\
In particular, $N' \in W\mathcal{A}$. With this, strictness of the monic $N' \to N$ in $W\mathcal{A}^\text{pre}$ is an instance of \ref{cor:1}. The construction of $N' \hookrightarrow N$ is thus complete. \\
The dual statement now follows from \ref{lemma:1}.
\end{proof}

\begin{subremark}
\namedlabel{rmk:7}{Remark \thesubsubsection}
With the notation of \ref{cstr:1}, any $f$ as in \ref{prop:2} factors as
\begin{equation*}
f: M \twoheadrightarrow M'' \twoheadrightarrow \text{Coim}f \cong \text{Im}f \hookrightarrow N' \hookrightarrow N
\end{equation*}
and it holds that $\text{Im}f \cong \text{Im}\!\left(M'' \to N'\right)$.
\end{subremark}

\begin{intermediate}
One can say more about this particular class of morphisms in $W\mathcal{A}^\text{pre}$.
\end{intermediate}

\begin{sublemma}
\namedlabel{lemma:3}{Lemma \thesubsubsection}
Let $i: M \hookrightarrow N$ be strict monic in $W\mathcal{A}^\text{pre}$ with $M \in W\mathcal{A}^\text{pre}_\text{udq}$ and $N \in W\mathcal{A}^\text{pre}_\text{uds}$, and consider the cokernel $\text{coker}(i)$ of $i$ in $W\mathcal{A}^\text{pre}$.
\begin{itemize}
\item[(1)] $\text{coker}(i)$ is also the cokernel of $i$ in $\text{Fun}(\Lambda,\mathcal{A})$.
\item[(2)] It holds that $\text{Coker}(i) \in W\mathcal{A}^\text{pre}_\text{uds}$. 
\item[(3)] $\text{Coker}(i) \in W\mathcal{A}$ when $N \in W\mathcal{A}$.
\end{itemize}
Note that $M \in W\mathcal{A}$ is automatic.
\end{sublemma}

\begin{proof}
Due to condition \eqref{eq:dagger} on the slope function, the assumptions imply that for every $\lambda$, the morphism $\text{gr}^W_\lambda i$ is strict monic. \\
From this, we obtain an isomorphism $W_{\leq\lambda}M/W_{<\lambda}M \to \left(W_{\leq\lambda}M + W_{<\lambda}N\right)/W_{<\lambda}N$. A comparison of the degrees yields $\text{deg}\!\left(W_{<\lambda}N + W_{\leq\lambda}M\right) - \text{deg}\!\left(W_{\leq\lambda}M\right) = \text{deg}\!\left(W_{<\lambda}N\right) - \text{deg}\!\left(W_{<\lambda}M\right)$ and the same is true for the ranks. Then again due to condition \eqref{eq:dagger}, also the epi-monic $W_{<\lambda}N/W_{<\lambda}M \to \left(W_{<\lambda}N + W_{\leq\lambda}M\right)/W_{\leq\lambda}M$ is an isomorphism, which means that the monic $i_\lambda: W_{<\lambda}N/W_{<\lambda}M \to W_{\leq\lambda}N/W_{\leq\lambda}M$ is strict. As $\text{Coker}\!\left(i_\lambda\right) \cong \text{Coker}\!\left(\text{gr}^W_\lambda i\right)$, we are done with (1) and (3). \\
Knowing that $\text{gr}^W_\lambda i$ is strict monic, one may as well conclude by applying the snake lemma in either abelian envelope of $\mathcal{A}$. \\
By (1), the sequence $0 \to M \stackrel{i}{\to} N \to \text{Coker}(i) \to 0$ is exact in $\text{Fun}(\mathcal{A},\Lambda)$ and induces for all $\lambda$ a short exact sequence
\begin{equation*}
0 \to \text{gr}^W_\lambda M \to \text{gr}^W_\lambda N \to \text{gr}^W_\lambda \text{Coker}(i) \to 0
\end{equation*}
in $\mathcal{A}$. Therefore, (2) results from \ref{lemma:8}. Again, (3) follows, because $W\mathcal{A}^\text{pre}_\text{udq}$ is closed under strict quotients in $W\mathcal{A}^\text{pre}$.
\end{proof}

\begin{intermediate}
Let us mention that the strategy of the proof of \ref{lemma:3} can be used to give a proof of \cite[Lemma 1.5.5, (1) $\implies$ (2)]{ARTICLE:1} in the case that the slope function there satisfies \eqref{eq:dagger}. \\

The following corollary completes the demonstration of \ref{thm:A} from the introduction.
\end{intermediate}

\begin{subcorollary}
\namedlabel{cor:2}{Corollary \thesubsubsection}
If $f: M \to N$ is a morphism in $W\mathcal{A}^\text{pre}$ with $M \in W\mathcal{A}^\text{pre}_\text{udq}$ and $N \in W\mathcal{A}^\text{pre}_\text{uds}$, then the cokernel of $f$ in $W\mathcal{A}^\text{pre}$ is computed in $\text{Fun}(\Lambda,\mathcal{A})$ and it lies in $W\mathcal{A}^\text{pre}_\text{uds}$. In particular, $\text{Coker}f \in W\mathcal{A}$ when $N \in W\mathcal{A}$. \\
Dually (\ref{lemma:1}), $\text{Ker}f \in W\mathcal{A}^\text{pre}_\text{udq}$, and it lies in $W\mathcal{A}$ when $M \in W\mathcal{A}$.
\end{subcorollary}

\begin{intermediate}
The kernel of $f$ is anyway computed in $\text{Fun}(\Lambda,\mathcal{A})$ because this is true for an arbitrary morphism in $W\mathcal{A}^\text{pre}$. One does not deduce this under use of \ref{lemma:1} because no equivalence $\left(W\mathcal{A}^\text{pre}\right)^\circ \simeq W\!\left(\mathcal{A}^\circ\right)^\text{pre}$ will extend to the ambient functor category (cokernels in $W\mathcal{A}^\text{pre}$ not being cokernels in $\text{Fun}(\Lambda,\mathcal{A})$ in general).
\end{intermediate}

\begin{proof}[Proof of \ref{cor:2}]
According to \ref{cor:1}, one has $\text{Im}f \in W\mathcal{A}$ and $M \to \text{Im}f$ is epi in the functor category. Then one applies \ref{lemma:3}.
\end{proof}

\subsection{An application}

The motivation behind the considerations of the preceding subsection comes the following theorem (\ref{thm:B} from the introduction), which is their intended application and as such the core of this work. It is inspired by and generalises Deligne's theorem \cite[Théorème 1.2.10]{ARTICLE:2} on the category of mixed Hodge structures being abelian, compare \ref{ex:5}. \\

We recall our standing assumptions: $\mathcal{A}$ is a quasi-abelian category equipped with a slope function $\mu$ (\ref{def:6}(3)), and $W\mathcal{A}^\text{pre}$, $W\mathcal{A}$ are as in \ref{def:1}. \\
In what follows, we will say that a functor between additive categories with kernels and cokernels is \textit{exact} if it preserves short exact sequences, and that it is \textit{strongly exact} if it preserves all kernels and cokernels.

\begin{subtheorem}
\namedlabel{thm:1}{Theorem \thesubsubsection}
$\text{}$
\begin{itemize}
\item[(0)] Any morphism in $W\mathcal{A}$ is strictly compatible and strictly co-compatible with the filtrations.
\item[(1)] The category $W\mathcal{A}$ is abelian.
\item[(2)] The full subcategory inclusion $W\mathcal{A} \hookrightarrow \text{Fun}(\Lambda,\mathcal{A})$ is strongly exact. In other words, the functor $W_{\leq \lambda}: W\mathcal{A} \to \mathcal{A}$ is strongly exact for every $\lambda \in \Lambda$. \\
Also the functor $\text{gr}^W_\lambda: W\mathcal{A} \to \mathcal{A}$ is strongly exact for every $\lambda \in \Lambda$.
\end{itemize}
\end{subtheorem}

\begin{proof}
(0) is a special case of \ref{prop:2} and we state it only for the sake of completeness, as it is in particular contained in (1) and (2). It follows from \ref{cor:2} that the additive category $W\mathcal{A}$ has kernels and cokernels, and also strong exactness of $W\mathcal{A} \hookrightarrow \text{Fun}(\Lambda,\mathcal{A})$ in (2) was established in \ref{cor:2}. Together with \ref{cor:1}, this implies (1). \\
Apply the nine lemma from homological algebra in an abelian envelope of $\mathcal{A}$ to see that exactness of $W\mathcal{A} \hookrightarrow \text{Fun}(\Lambda,\mathcal{A})$ is equivalent to exactness of $\text{gr}^W_\lambda: W \mathcal{A} \to \mathcal{A}$ for all $\lambda$. Since $W\mathcal{A}$ is abelian, an exact functor out of $W\mathcal{A}$ is already strongly exact.
\end{proof}

\begin{intermediate}
\ref{thm:1} asserts in particular that also the inclusion $W\mathcal{A} \hookrightarrow W\mathcal{A}^\text{pre}$ is strongly exact. It follows that the canonical functor $W\mathcal{A} \to \mathcal{A}$ is strongly exact, so the underlying morphism of a morphism in $W\mathcal{A}$ is strict in $\mathcal{A}$. Moreover, $W\mathcal{A} \to \mathcal{A}$ reflects kernels, cokernels and isomorphisms.
\end{intermediate}

\begin{subremark}
\namedlabel{rmk:4}{Remark \thesubsubsection}
For contrast, we look at the failure of (strong) exactness of $W\mathcal{A}^\text{pre} \hookrightarrow \text{Fun}(\Lambda,\mathcal{A})$.
\begin{itemize}
\item[(1)] Although the canonical functor $W\mathcal{A}^\text{pre} \to \mathcal{A}$ is always strongly exact, this does, with the exception of the trivial case, never apply to $W\mathcal{A}^\text{pre} \hookrightarrow \text{Fun}(\Lambda,\mathcal{A})$ (take the example from \ref{rmk:1}). If $\mathcal{A}$ is abelian, the inclusion $W\mathcal{A}^\text{pre} \hookrightarrow \text{Fun}(\Lambda,\mathcal{A})$ is at least exact. There are quasi-abelian categories for which this fails to hold. \\
We give an example which will also be used for later reference. Let $\mathcal{B}$ be the abelian category of finite-dimensional vector spaces over a field $k$ and let $\mathcal{B}_1$ be the quasi-abelian category of filtered $k$-vector spaces. The following sequence in $W\!\left(\mathcal{B}_1\right)^\text{pre}$ with the canonical morphisms constitutes a short exact sequence in $W\!\left(\mathcal{B}_1\right)^\text{pre}$ which is not exact in the ambient functor category:
\[\begin{tikzcd}
0 & A & B & C & 0
\arrow[from=1-1, to=1-2]
\arrow[from=1-2, to=1-3]
\arrow[from=1-3, to=1-4]
\arrow[from=1-4, to=1-5]
\end{tikzcd}\]
\[\begin{tikzcd}[sep={1.75cm,between origins}]
& 0 & 0 & k\begin{pmatrix} 1 \\ 1 \end{pmatrix} \\
{A =} & 0 & 0 & 0 \\
& 0 & 0 & 0
\arrow[hook, from=1-2, to=1-3]
\arrow[hook, from=1-3, to=1-4]
\arrow[hook, from=2-2, to=1-2]
\arrow[hook, from=2-2, to=2-3]
\arrow[hook, from=2-3, to=1-3]
\arrow[hook, from=2-3, to=2-4]
\arrow[hook, from=2-4, to=1-4]
\arrow[hook, from=3-2, to=2-2]
\arrow[hook, from=3-2, to=3-3]
\arrow[hook, from=3-3, to=2-3]
\arrow[hook, from=3-3, to=3-4]
\arrow[hook, from=3-4, to=2-4]
\end{tikzcd}
\begin{tikzcd}[sep={1.75cm,between origins}]
& 0 & k\begin{pmatrix} 1 \\ 0 \end{pmatrix} & {k^2} \\
{B =} & 0 & 0 & k\begin{pmatrix} 0 \\ 1 \end{pmatrix} \\
& 0 & 0 & 0
\arrow[hook, from=1-2, to=1-3]
\arrow[hook, from=1-3, to=1-4]
\arrow[hook, from=2-2, to=1-2]
\arrow[hook, from=2-2, to=2-3]
\arrow[hook, from=2-3, to=1-3]
\arrow[hook, from=2-3, to=2-4]
\arrow[hook, from=2-4, to=1-4]
\arrow[hook, from=3-2, to=2-2]
\arrow[hook, from=3-2, to=3-3]
\arrow[hook, from=3-3, to=2-3]
\arrow[hook, from=3-3, to=3-4]
\arrow[hook, from=3-4, to=2-4]
\end{tikzcd}\]
\[\begin{tikzcd}[sep={2.25cm,between origins}]
& 0 & k^2/k\begin{pmatrix} 1 \\ 1 \end{pmatrix} & k^2/k\begin{pmatrix} 1 \\ 1 \end{pmatrix} \\
{C =} & 0 & k^2/k\begin{pmatrix} 1 \\ 1 \end{pmatrix} & k^2/k\begin{pmatrix} 1 \\ 1 \end{pmatrix} \\
& 0 & 0 & 0
\arrow[hook, from=1-2, to=1-3]
\arrow[hook, from=1-3, to=1-4]
\arrow[hook, from=2-2, to=1-2]
\arrow[hook, from=2-2, to=2-3]
\arrow[hook, from=2-3, to=1-3]
\arrow[hook, from=2-3, to=2-4]
\arrow[hook, from=2-4, to=1-4]
\arrow[hook, from=3-2, to=2-2]
\arrow[hook, from=3-2, to=3-3]
\arrow[hook, from=3-3, to=2-3]
\arrow[hook, from=3-3, to=3-4]
\arrow[hook, from=3-4, to=2-4]
\end{tikzcd}\]
\item[(2)] Exactness of $W\mathcal{A}^\text{pre} \hookrightarrow \text{Fun}(\Lambda,\mathcal{A})$ in the case that $\mathcal{A}$ is abelian implies that in this case, $W\mathcal{A}$ is closed under extensions in $W\mathcal{A}^\text{pre}$. In general, it follows from \ref{lemma:3} that if $0 \to M \to N \to P \to 0$ is a short exact sequence in $W\mathcal{A}^\text{pre}$ with $M,P \in W\mathcal{A}$ and $N \in W\mathcal{A}^\text{pre}_\text{uds}$ or, dually, $N \in W\mathcal{A}^\text{pre}_\text{udq}$, then $N \in W\mathcal{A}$. \\
The author does not know whether it is true for an arbitrary quasi-abelian category $\mathcal{A}$ that $W\mathcal{A}$ is closed under extensions in $W\mathcal{A}^\text{pre}$. At least one readily obtains extension-closedness inside the functor category, see \ref{lemma:6}.
\end{itemize}
\end{subremark}

\begin{subremark}
We give another proof of \ref{thm:1}, using an argument which Simpson \cite[Lemma 1.3]{MISC:3} employed to show abelianess of the category of mixed twistor structures.
\begin{itemize}
\item[(1)] First, we translate the proof of \cite[Lemma 1.3]{MISC:3} to our setup. For this purpose, recall \ref{rmk:6}(4). The strategy is to show that kernels and cokernels in $W\mathcal{A}$ are computed in the left abelian envelope $\mathcal{L}_\mathcal{A}$ of $\mathcal{A}$. Both our proof of strict compatibility of morphisms in $W\mathcal{A}$ with the filtrations and the argument here proceed by induction on the number of breaks of the filtrations on domain and codomain, however while we disposed of the lowest break in order to apply the induction hypothesis, the focus will be on the highest break in the following. \\
Given a morphism $f: (M,W) \to (M',W')$ in $W\mathcal{A}$, define $(M'',W'')$ to be the cokernel of $f$ in $W\mathcal{L}_\mathcal{A}^\text{pre}$, that is $M'' \defeq \text{Coker}_{\mathcal{L}_\mathcal{A}}f$ and $W''_{\leq\lambda} M'' \defeq \frac{W'_{\leq\lambda}M' + f(M)}{f(M)} = \frac{W'_{\leq\lambda}M'}{W'_{\leq\lambda}M' \cap f(M)}$, where everything is computed in $\mathcal{L}_\mathcal{A}$. We do induction on the sum of the number of breaks of the filtrations $W$ and $W'$ to show that $(M'',W'')$ is in $W\mathcal{A}$. If both filtrations break at most once, then $f$ is either zero or a morphism between pure structures of the same weight, so that in any case $(M'',W'') \in W\mathcal{A}$. \\
In the induction step, let $\lambda \in \Lambda$ be minimal such that both $W_{\leq\lambda}M = M$ and $W'_{\leq\lambda}M' = M'$. In $\mathcal{L}_\mathcal{A}$, it holds that $\text{gr}^{W''}_\lambda M'' = \frac{M'}{W'_{<\lambda}M' + f(M)}$ is the cokernel of $\text{gr}_\lambda f$. Since $\text{gr}_\lambda f$ is zero or a morphism between semistable objects of the same slope, it is strict, and hence $\text{gr}^{W''}_\lambda M''$ lies in $\mathcal{A}$ and is the cokernel of $\text{gr}_\lambda f$ in $\mathcal{A}$. Moreover, it belongs to $\mathcal{A}(\lambda)$. \\
Let $V \defeq f^{-1}\!\left(W'_{<\lambda}M'\right)$, computed equivalently in $\mathcal{A}$ or in $\mathcal{L}_\mathcal{A}$. If $K$ denotes the kernel of $\text{gr}_\lambda f$ (in $\mathcal{A}$ or $\mathcal{L}_\mathcal{A}$), we have a short exact sequence in $\mathcal{A}$
\begin{equation*}
0 \to W_{<\lambda}M \to V \to K \to 0
\end{equation*}
and furthermore $K \in \mathcal{A}(\lambda)$. \\
By induction, the cokernel of $f_{<\lambda}: \left(W_{<\lambda}M,W_{<\lambda}M \cap W\right) \to \left(W'_{<\lambda}M',W'_{<\lambda}M' \cap W'\right)$ in $W\mathcal{L}_\mathcal{A}^\text{pre}$ lies in $W\mathcal{A}$, note that at least one of $W$ or $W'$ breaks at $\lambda$. This implies $L \defeq \text{Coker}_{\mathcal{L}_\mathcal{A}}\!\left(f_{<\lambda}\right) \in \mathcal{A}$ and that the morphism $K = V/W_{<\lambda}M \to L$ induced by $f$ is zero. So $f(V) = f\!\left(W_{<\lambda}M\right)$ in $\mathcal{L}_\mathcal{A}$. \\
At the same time, $f(V) = f\!\left(f^{-1}\!\left(W'_{<\lambda}M'\right)\right) = W'_{<\lambda}M' \cap f(M) = \text{Ker}\!\left(W'_{<\lambda}M' \to W''_{<\lambda}M''\right)$ in $\mathcal{L}_\mathcal{A}$. Therefore, $L = W''_{<\lambda}M''$. In total, $W''_{<\lambda}M''$ with the filtration induced from $M''$ is in $W\mathcal{A}$. \\
Now
\begin{equation*}
0 \to W''_{<\lambda}M'' \to M'' \to \text{gr}^{W''}_\lambda M'' \to 0
\end{equation*}
is a short exact sequence in $\mathcal{L}_\mathcal{A}$ with $W''_{<\lambda}M'', \text{gr}^{W''}_\lambda M'' \in \mathcal{A}$. Since $\mathcal{A}$ is stable under extensions in $\mathcal{L}_\mathcal{A}$, we get $M'' \in \mathcal{A}$ and are done. \\
With the necessary modifications, we can run the argument also in the right abelian envelope $\mathcal{R}_\mathcal{A}$ of $\mathcal{A}$. This allows to deduce by duality that also kernels in $W\mathcal{A}$ are computed in $W\mathcal{L}_\mathcal{A}^\text{pre}$. \\
In particular, any morphism of $W\mathcal{A}$ has the property that its underlying morphism in $\mathcal{A}$ is strict. Therefore, in order to establish abelianess of $W\mathcal{A}$ as well as strong exactness of $W\mathcal{A} \hookrightarrow \text{Fun}(\Lambda,\mathcal{A})$, it now suffices to show that a morphism of $W\mathcal{A}$ whose underlying morphism in $\mathcal{A}$ is an isomorphism is strictly compatible (or, inside $W\mathcal{A}$ equivalently, strictly co-compatible) with the filtrations. But in fact, strict compatibility of any morphism in $W\mathcal{A}$ can already be read off from the above induction. Indeed, keeping our notation, $f\!\left(W_{\leq\lambda}M\right) = f(M) = M' \cap f(M) = W'_{\leq\lambda}M' \cap f(M)$ holds, $f_{<\lambda}$ may be assumed by induction to be strictly compatible with the filtrations and we have seen that $\text{Im}f_{<\lambda} = f\!\left(W_{<\lambda}M\right) = W'_{<\lambda}M' \cap f(M)$.
\item[(2)] Second, we indicate how one can generalise this argument to obtain a proof of \ref{prop:2}. Namely, starting more generally from a morphism $f: (M,W) \to (M',W')$ in $W\mathcal{A}^\text{pre}$ with $(M,W) \in W\mathcal{A}^\text{pre}_\text{udq}$ and $(M',W') \in W\mathcal{A}^\text{pre}_\text{uds}$, but otherwise keeping notation, one sees that all nonzero quotients of $K$ have slope $\geq \lambda$ by looking at the short exact sequence
\begin{equation*}
0 \to K \to \text{gr}^W_\lambda M \to \text{Coim}_\mathcal{A}\!\left(\text{gr}_\lambda f\right) \to 0
\end{equation*}
in $\mathcal{A}$. Assuming by induction that all nonzero subobjects of $L$ have slope $< \lambda$, one again finds that the morphism $K \to L$ induced by $f$ is zero and concludes as before that the cokernel of $f$ in $W\mathcal{A}^\text{pre}$ is computed in $W\mathcal{L}_\mathcal{A}^\text{pre}$ and that $f$ is strictly compatible with the filtrations (and hence that the cokernel of $f$ lies in $W\mathcal{A}^\text{pre}_\text{uds}$ as required for the induction).
\end{itemize}
\end{subremark}

\subsection{Some remarks}

First, we add two lemmata whose statements we have mentioned but which we have not used in our proofs.

\begin{sublemma}
\namedlabel{lemma:5}{Lemma \thesubsubsection}
Let $\text{UDS}(X) \hookrightarrow X$ denote the universal destabilising subobject of $X \in \mathcal{A}$ and $X \twoheadrightarrow \text{UDQ}(X)$ the universal destabilising quotient. Write $\langle O\rangle_\mathcal{A}$ for the envelope of a collection of objects in $\mathcal{A}$, that is the smallest extension-stable full subcategory of $\mathcal{A}$ containing the objects of $O$. For $X \in \mathcal{A}$, it holds that
\begin{equation*}
X \in \left\langle\bigcup_{\lambda \leq \nu} \mathcal{A}(\lambda)\right\rangle_{\!\mathcal{A}} \iff \mu\!\left(\text{UDS}(X)\right) \leq \nu \iff F^{>\nu}X = 0
\end{equation*}
and
\begin{equation*}
X \in \left\langle\bigcup_{\lambda \geq \nu} \mathcal{A}(\lambda)\right\rangle_{\!\mathcal{A}} \iff \mu\!\left(\text{UDQ}(X)\right) \geq \nu \iff F^{\geq\nu}X = X.
\end{equation*}
\end{sublemma}

\begin{proof}
It remains to justify the respective first implications. To this end, we need to show that the property of an object of $\mathcal{A}$ that its subobjects respectively quotients have bounded slope is stable under extensions in $\mathcal{A}$. \\
The argument is similar to the proof of \cite[Lemma 1.3.7(7)]{ARTICLE:1}: Let $0 \to X \to Y \to Z \to 0$ be a short exact sequence in $\mathcal{A}$ where $\mu\!\left(\text{UDS}(X)\right) \leq \nu$ and $\mu\!\left(\text{UDS}(Z)\right) \leq \nu$. Assume by contradiction that the universal destablising subobject of $Y$ had slope $> \nu$. Then the composite $\text{UDS}(Y) \hookrightarrow Y \to Z$ would have image of slope $> \nu$, hence be zero. Consequently, $\text{UDS}(Y) \hookrightarrow Y$ would factor through $X \to Y$, exhibiting $\text{UDS}(Y)$ as a strict subobject of $X$, which is not possible. \\
Together with duality, concretely
\begin{equation*}
\left[\left\langle\bigcup_{\lambda \geq \nu} \mathcal{A}(\lambda)\right\rangle_{\!\mathcal{A}}\right]^\circ = \left\langle\bigcup_{-\lambda \leq -\nu} \mathcal{A}^\circ(-\lambda)\right\rangle_{\!\mathcal{A}^\circ},
\end{equation*}
this proves the claim.
\end{proof}

\begin{intermediate}
\ref{lemma:5} demonstrates how slopes control when nonzero morphisms between objects are admissible. Namely, it shows, again, that $\text{UDS}(X) \leq \nu$ holds if and only if $X$ does not admit a nonzero morphism from a semistable object of slope $> \nu$ and that $\text{UDQ}(X) \geq \nu$ holds if and only if $X$ does not admit a nonzero morphism to a semistable object of slope $< \nu$.
\end{intermediate}

\begin{sublemma}
\namedlabel{lemma:6}{Lemma \thesubsubsection}
$W\mathcal{A}^\text{pre}$, $W\mathcal{A}^\text{pre}_\text{uds}$, $W\mathcal{A}^\text{pre}_\text{udq}$ and $W\mathcal{A}$ are closed under extensions in $\text{Fun}(\Lambda,\mathcal{A})$.
\end{sublemma}

\begin{proof}
Let $0 \to M \to N \to P \to 0$ be a short exact sequence in $\text{Fun}(\Lambda,\mathcal{A})$ where $M, P \in W\mathcal{A}^\text{pre}$. For $\lambda \leq \lambda_+$, we have the following commutative diagram in $\mathcal{A}$ with exact rows:
\[\begin{tikzcd}[column sep=small]
0 & {W_{\leq\lambda}M} & {N(\lambda)} & {W_{\leq\lambda}P} & 0 \\
0 & {W_{\leq\lambda_+}M} & {N\!\left(\lambda_+\right)} & {W_{\leq\lambda_+}P} & 0
\arrow[from=1-1, to=1-2]
\arrow[from=1-2, to=1-3]
\arrow[hook, from=1-2, to=2-2]
\arrow[from=1-3, to=1-4]
\arrow["\alpha", from=1-3, to=2-3]
\arrow[from=1-4, to=1-5]
\arrow[hook, from=1-4, to=2-4]
\arrow[from=2-1, to=2-2]
\arrow[from=2-2, to=2-3]
\arrow[from=2-3, to=2-4]
\arrow[from=2-4, to=2-5]
\end{tikzcd}\]
The snake lemma in the left abelian envelope $\mathcal{L}_\mathcal{A}$ of $\mathcal{A}$ shows that $\text{Ker}_\mathcal{A}\alpha = \text{Ker}_{\mathcal{L}_\mathcal{A}}\alpha = 0$, so $\alpha$ is monic in $\mathcal{A}$, and it further yields a short exact sequence
\begin{equation*}
0 \to \text{gr}^W_{\lambda_+}M \to \text{Coker}_{\mathcal{L}_\mathcal{A}}\alpha \to \text{gr}^W_{\lambda_+}P \to 0
\end{equation*}
in $\mathcal{L}_\mathcal{A}$. Since $\mathcal{A}$ is closed under extensions in $\mathcal{L}_\mathcal{A}$, it follows that $\text{Coker}_{\mathcal{L}_\mathcal{A}}\alpha \in \mathcal{A}$ and that hence $\text{Coker}_{\mathcal{L}_\mathcal{A}}\alpha \to \text{Coker}_\mathcal{A}\alpha$ is an isomorphism. This implies that $\alpha$ is strict in $\mathcal{A}$: From the five lemma in $\mathcal{L}_\mathcal{A}$, one deduces that $\text{Im}_{\mathcal{L}_\mathcal{A}}\alpha \to \text{Im}_\mathcal{A}\alpha$ is an isomorphism, and it anyway holds that $\text{Im}_{\mathcal{L}_\mathcal{A}}\alpha = \text{Coim}_\mathcal{A}\alpha$. In summary, $\alpha$ is strict monic. \\
Moreover, $\text{Coker}~\alpha = 0$ whenever $\text{Coker}\!\left(W_{\leq\lambda}M \hookrightarrow W_{\leq\lambda_+}M\right) = \text{Coker}\!\left(W_{\leq\lambda}P \hookrightarrow W_{\leq\lambda_+}P\right) = 0$, and $N(\lambda) = 0$ holds for sufficiently small $\lambda$. Therefore, $N \in W\mathcal{A}^\text{pre}$ (and this conclusion remains valid for arbitrary quasi-abelian $\mathcal{A}$). \\
With this, extension-closedness of $W\mathcal{A}^\text{pre}_\text{uds}$, $W\mathcal{A}^\text{pre}_\text{udq}$ and consequently of $W\mathcal{A}$ in $\text{Fun}(\Lambda,\mathcal{A})$ follows from \ref{lemma:5}.
\end{proof}

\begin{intermediate}
Next, we comment on the existence and uniqueness of weight filtrations. Then, to conclude this section, we discuss a link to Bondarko's weight structures. \\

From the characterisation of exact slope filtrations \cite[Theorem 1.5.9]{ARTICLE:1}, we will use the implication $(1) \implies (4)$, and the converse in the case that the quasi-abelian category $\mathcal{A}$ is abelian.
\end{intermediate}

\begin{subremark}
\namedlabel{rmk:2}{Remark \thesubsubsection}
$\text{}$
\begin{itemize}
\item[(1)] The example in \ref{rmk:4}(1) can also be read as a short exact sequence in the category $\mathcal{B}_2$, the quasi-abelian category of objects of $\mathcal{B}$ equipped with two descending filtrations, which we will endow with a slope function in \ref{ex:5}. In fact, it is then a short exact sequence of semistable objects of strictly increasing slope. Other instances of this phenomenon are \ref{rmk:3}(1) as well as \cite[(1.11) in Examples 1.5.3(1)]{ARTICLE:1}. \\
Whenever a semistable object is the extension of semistable objects of different slopes, we can put at least two different weight filtrations on it. Therefore, the above examples show that it is not always the case that the weight filtration is unique in the sense that, up to isomorphism, there is at most one filtration $W$ on any given $M \in \mathcal{A}$ such that $(M,W) \in W\mathcal{A}$.
\item[(2)] At least if the slope filtration is strongly exact, in that the functors $F^{\geq\lambda}_\mu: \mathcal{A} \to \mathcal{A}$ are strongly exact for all $\lambda \in \Lambda$ (which is restrictive, note (4)), then the weight filtration is unique. Indeed, consider $(M,W), \left(M,W'\right) \in W\mathcal{A}$. Let $\nu$ be the smallest break of the filtration $W$ and let $\lambda$ be the smallest break of $W'$ such that $W_{\leq\nu}M \hookrightarrow M$ factors through $W'_{\leq\lambda}M \hookrightarrow M$. Then there is a nonzero morphism $W_{\leq\nu}M \to \text{gr}^{W'}_\lambda M$, so strong exactness of the slope filtration implies that $\nu = \lambda$ (because the image of this morphism is at the same time of slope $\nu$ and of slope $\lambda$, compare \cite[Theorem 1.5.9, (1) $\implies$ (4)]{ARTICLE:1}). Interchanging the roles of $W$ and $W'$, we get that their smallest filtration steps coincide. Hence we can proceed with $\left(M/W_{\leq\nu}M,W\right), \left(M/W_{\leq\nu}M,W'\right) \in W\mathcal{A}$ and inductively, we find that there is an isomorphism $(M,W) \cong (M,W')$ in $W\mathcal{A}$ induced by the identity of $M$.
\item[(3)] The argument from (2) applied to a morphism $f: M \to M'$ in $\mathcal{A}$ in place of the identity $M \to M$ shows that if the slope filtration on $\mathcal{A}$ is strongly exact and $M$, $M'$ admit a weight filtration, then $f$ respects the weight filtrations. In this situation, the canonical functor $W\mathcal{A} \to \mathcal{A}$ is fully faithful and an equivalence onto the full subcategory $\mathcal{A}' \subset \mathcal{A}$ of objects of $\mathcal{A}$ admitting a weight filtration. In particular, $\mathcal{A}'$ is abelian and closed under kernels and cokernels in $\mathcal{A}$. On $\mathcal{A}'$, there is a functorial weight filtration $W: \mathcal{A}' \to \text{Fun}(\Lambda,\mathcal{A})$. \\
The slope filtration on $\mathcal{A}$ restricts to a slope filtration on $\mathcal{A}'$, which is split by $W$ (see (5)). The increasing family of full subcategories $\mathcal{A}'_{\leq\lambda} \subset \mathcal{A}'$ consisting of the objects $A \in \mathcal{A}'$ with $W_{\leq\lambda}A = A$ forms a weight filtration of $\mathcal{A}'$ in the sense of \cite[Definition 2.1.1]{ARTICLE:7}.
\item[(4)] Due to condition \eqref{eq:dagger} on the slope function, a strongly exact slope filtration forces the quasi-abelian category $\mathcal{A}$ to be abelian. Namely, for every epi-monic $f$ in $\mathcal{A}$ and all $\lambda \in \Lambda$, $\text{gr}_F^\lambda f$ is then strict epi and $F^{\geq\lambda}f$ is monic, so by descending induction on $\lambda$, one concludes under use of the snake lemma in the left abelian envelope of $\mathcal{A}$ that $f$ is an isomorphism.
\item[(5)] If $\mathcal{A}$ is abelian and the slope filtration is (strongly) exact, any weight filtration on $X \in \mathcal{A}$ is opposed to and consequently splits the slope filtration on $X$ in the sense of and according to \cite[Proposition 1.2.5(ii), 1.2.6]{ARTICLE:2}, refer to \ref{ex:4}. In this situation, a weight filtration, by (3) automatically functorial, can only exist on every object of $\mathcal{A}$ if the slope filtration is \textit{split} \cite[Definition 1.5.11]{ARTICLE:1}, that is if $\text{gr} \cong \text{id}$ as functors. \\
For instance, a mixed Hodge structure which arises as a non-split extension of pure Hodge structures does not admit a weight filtration with respect to the slope function on the category of mixed Hodge structures from \cite[Example 1.5.7]{ARTICLE:1}. An example can be constructed from \ref{rmk:4}(1).
\item[(6)] We elaborate on the condition of being split: In the case that $\mathcal{A}$ is abelian, for every slope function $\mu$ on $\mathcal{A}$, $-\mu$ is again a slope function. It holds that the slope filtration corresponding to $\mu$ splits if and only if the slope filtration corresponding to $-\mu$ splits if and only if both the slope filtrations for $\mu$ and for $-\mu$ are exact. Namely, the slope filtration for $\mu$ is exact if and only if semistability with respect to $\mu$ implies semistability with respect to $-\mu$. \\
Whenever the slope filtration for $-\mu$ is exact, by reindexing $\lambda \to - \lambda$, one obtains a functorial weight filtration for $\mu$. If also $\mu$ is exact, the latter filtration is opposed to and functorially splits the slope filtration for $\mu$. \\
Conversely, if the slope filtration for $\mu$ splits, it is in particular exact, hence the graded pieces are also semistable with respect to $-\mu$. From the functorial splitting, one builds the slope filtration for $-\mu$, which is again exact.
\end{itemize}
\end{subremark}

\begin{subremark}
Let $\mathcal{A}$ be an abelian category equipped with a $\mathbb{Z}$-valued slope function $\mu$ and consider the bounded derived categories $\mathcal{D}^b(\mathcal{A})$ and $\mathcal{D}^b\!\left(W\mathcal{A}\right)$ with their canonical $t$-structures. We use the homological convention. Let $\mathcal{D}$ be either $\mathcal{D}^b(\mathcal{A})$ or $\mathcal{D}^b\!\left(W\mathcal{A}\right)$. \\
Here, we relate weight filtrations in our sense to weight filtrations coming from transversal weight structures on $\mathcal{D}$ in the sense of \cite{ARTICLE:7}. On the one hand, we will see in (1) that for $\mathcal{D} = \mathcal{D}^b(\mathcal{A})$, the interface is narrow. On the other hand, we will demonstrate in (2) and (3) that under less restrictive conditions, it is possible to reconstruct the filtration $W$ from a transversal weight structure on $\mathcal{D} = \mathcal{D}^b\!\left(W\mathcal{A}\right)$. \\
The picture could be different for bounded weight structures on $\mathcal{D}$ which, rather than being transversal to the $t$-structure, satisfy the weaker condition that the corresponding weight spectral sequence degenerates \cite[Section 3.2]{ARTICLE:7}. We have not investigated this.
\begin{itemize}
\item[(1)] Suppose that a weight filtration on $\mathcal{A}$ with respect to the slope function $\mu$ is obtained as in \cite[Theorem 1.2.1]{ARTICLE:7} from a \textit{bounded weight structure} on $\mathcal{D} = \mathcal{D}^b(\mathcal{A})$ \cite[Notation]{ARTICLE:7} which is \textit{transversal} \cite[Definition 1.2.2]{ARTICLE:7} to the $t$-structure. That is, suppose there exists is a bounded weight structure $\left(\mathcal{D}_{w \leq 0},\mathcal{D}_{w \geq 0}\right)$ on $\mathcal{D}^b(\mathcal{A})$ transversal to the $t$-structure such that $\mathcal{D}^{t = 0} \cap \mathcal{D}_{w = 0}[\lambda] \simeq \mathcal{A}(\lambda)$. Then the family of subcategories $\left(\mathcal{A}(\lambda)\right)_\lambda$ is a \textit{semi-orthogonal family} \cite[Definition 1.1.4]{ARTICLE:7} in $\mathcal{D}^b(\mathcal{A})$. We claim that this implies that the abelian category $\mathcal{A}$ is \textit{semisimple}, in that every short exact sequence in $\mathcal{A}$ splits. \\
If $\left(\mathcal{A}(\lambda)\right)_\lambda$ is semi-orthogonal in $\mathcal{D}^b(\mathcal{A})$, we have $\text{Hom}_\mathcal{A}(A,B) = 0$ for $A \in \mathcal{A}(\lambda)$, $B \in \mathcal{A}(\nu)$ whenever $\nu > \lambda$ and $\text{Ext}_\mathcal{A}^1(A,B) = 0$ whenever $\nu \geq \lambda$. As a consequence, the slope filtration is exact and on every object of $\mathcal{A}$, it splits. Hence by \ref{rmk:2}(3), there is a unique functorial weight filtration on $\mathcal{A}$. Due to exactness of the slope filtration, this weight filtration is opposed to and splits the slope filtration in fact functorially (\ref{rmk:2}(5)). But then $\mathcal{A}$ is semisimple, as all categories $\mathcal{A}(\lambda)$ are. \\
Put differently, since $\mu$ is exact, the functorial weight filtration obtained from \cite[Theorem 1.2.1(iii)]{ARTICLE:7} is the negative of the slope filtration corresponding to $-\mu$ on $\mathcal{A}$. As this filtration is exact by \cite[Proposition 1.2.4]{ARTICLE:7}, \ref{rmk:2}(6) applies.
\end{itemize}
Instead, the construction of transversal weight structures allows to recover certain exact slope filtrations on, not necessarily semisimple, abelian categories. To demonstrate this, consider the condition:
\begin{equation}
\setlength{\jot}{0pt}
\label{eq:lozenge}
\tag{$\lozenge$}
\mu \text{ is exact and  } \mathcal{A}(\lambda) \text{ is semisimple for all } \lambda.
\end{equation}
This is a weaker condition than semi-orthogonality of $\left(\mathcal{A}(\lambda)\right)_\lambda$ in (1) (in fact strictly weaker, as one gets an example from \ref{rmk:2}(5)). It follows from \cite[Proposition 2.1.3]{ARTICLE:7} that \eqref{eq:lozenge} is equivalent to semi-orthogonality of the family $\left(\mathcal{A}(-\lambda)\right)_\lambda$. This opens up yet another perspective on the argument in (1): As soon as $\left(\mathcal{A}(\lambda)\right)_\lambda$ is semi-orthogonal, so is $\left(\mathcal{A}(-\lambda)\right)_\lambda$, and semisimplicity of $\mathcal{A}$ in this case is a direct consequence.
\begin{itemize}
\item[(2)] Define $\mathcal{D}_{s \leq 0} \defeq \left\langle\bigcup_n\bigcup_{\lambda \geq n} \mathcal{A}(\lambda)[n]\right\rangle_{\mathcal{D}^b(\mathcal{A})}$ and $\mathcal{D}_{s \geq 0} \defeq \left\langle\bigcup_n\bigcup_{\lambda \leq n} \mathcal{A}(\lambda)[n]\right\rangle_{\mathcal{D}^b(\mathcal{A})}$. Assuming \eqref{eq:lozenge}, one reads off from the proof of \cite[Theorem 1.2.1]{ARTICLE:7} that the pair $\left(\mathcal{D}_{s \leq 0},\mathcal{D}_{s \geq 0}\right)$ is a bounded weight structure on $\mathcal{D}^b(\mathcal{A})$ which is transversal to the $t$-structure. \\
If we put $\mathcal{D}_{s \leq \lambda} \defeq \mathcal{D}_{s\leq 0}[\lambda]$, $\mathcal{D}_{s \geq \lambda} \defeq \mathcal{D}_{s \geq 0}[\lambda]$ and $\mathcal{D}_{s = \lambda} \defeq \mathcal{D}_{s \leq \lambda} \cap \mathcal{D}_{s \geq \lambda}$, it holds that $\mathcal{D}^{t = 0} \cap \mathcal{D}_{s = \lambda} \simeq \mathcal{A}(-\lambda)$, and the resulting functorial filtration on $\mathcal{D}^{t = 0} \simeq \mathcal{A}$ from \cite[Theorem 1.2.1(iii),(iii')]{ARTICLE:7} is the negative of the slope filtration on $\mathcal{A}$. \\
Explicitly, one has
\begin{equation*}
\mathcal{D}_{s \leq 0} = \left\{X \in \mathcal{D}^b(\mathcal{A}) \mid \forall n \in \mathbb{Z}: H_nX \in \left\langle\bigcup_{\lambda \geq n} \mathcal{A}(\lambda)\right\rangle_{\!\mathcal{A}}\right\}
\end{equation*}
and
\begin{equation*}
\mathcal{D}_{s \geq 0} = \left\{X \in \mathcal{D}^b(\mathcal{A}) \mid \forall n \in \mathbb{Z}: H_nX \in \left\langle\bigcup_{\lambda \leq n} \mathcal{A}(\lambda)\right\rangle_{\!\mathcal{A}}\right\}.
\end{equation*}
To get containment of the respective left-hand in the respective right-hand side, observe that exactness of the slope filtration is equivalent to $\left\langle\bigcup_{\lambda \geq n} \mathcal{A}(\lambda)\right\rangle_{\!\mathcal{A}}$ and $\left\langle\bigcup_{\lambda \leq n} \mathcal{A}(\lambda)\right\rangle_{\!\mathcal{A}}$ being closed under both subobjects and quotients in $\mathcal{A}$, and employ the long exact homology sequences. For the reverse inclusions, do induction on the number of nonzero homology groups and use the distinguished triangles for the canonical truncations. \\
\ref{lemma:5} explains the shape of this weight structure and shows that \cite[Proposition 2.1.3]{ARTICLE:7} is applicable. Namely, $\mathcal{D}_{s \leq \lambda} \cap \mathcal{D}^{t = 0} \simeq \left\{X \in \mathcal{A} \mid \mu\!\left(\text{UDQ}(X)\right) \geq -\lambda\right\} \eqdef \mathcal{A}_{\leq \lambda}$ as well as $\mathcal{D}_{s \geq \lambda} \cap \mathcal{D}^{t = 0} \simeq \left\{X \in \mathcal{A} \mid \mu\!\left(\text{UDS}(X)\right) \leq -\lambda\right\} \eqdef \mathcal{A}_{\geq \lambda}$ are abelian subcategories of $\mathcal{D}^{t = 0} \simeq \mathcal{A}$, closed under both subobjects and quotients. For every $\lambda$, $F^{\geq -\lambda}$ is an exact functor $\mathcal{A} \to \mathcal{A}_{\leq \lambda}$ which is right adjoint to the inclusion $\mathcal{A}_{\leq \lambda} \hookrightarrow \mathcal{A}$. Therefore, the sequence of subcategories $\mathcal{A}_{\leq \lambda} \subset \mathcal{A}$ defines a weight filtration of $\mathcal{A}$ in the sense of \cite[Definition 2.1.1]{ARTICLE:7}. Also the functor $\mathcal{A} \to \mathcal{A}_{\geq \lambda}$, $X \mapsto X/F^{> - \lambda}X$ is exact and it is left adjoint to the inclusion $\mathcal{A}_{\geq \lambda} \hookrightarrow \mathcal{A}$.
\item[(3)] Finally, let $\mathcal{A}$ be any quasi-abelian category with a $\mathbb{Z}$-valued slope function $\mu$ and suppose that the abelian categories $\mathcal{A}(\lambda)$ are semisimple. Then $\mu' \defeq (-\mu) \circ (W\mathcal{A} \to \mathcal{A})$ is an exact slope function on the abelian category $W\mathcal{A}$ (with corresponding slope filtration $F_{\mu'}^{\geq\lambda}(M,W) =\left(W_{\leq-\lambda}M, W_{\leq-\lambda}M \cap W\right)$, compare \cite[Example 1.5.7]{ARTICLE:1}) and the categories $\mathcal{A}'(\lambda)$ are again semisimple. If we perform the construction of the weight structure $\left(\mathcal{D}_{s \leq 0}, \mathcal{D}_{s \geq 0}\right)$ from (2) on $\mathcal{D} = \mathcal{D}^b\!\left(W\mathcal{A}\right)$ for the slope function $\mu'$ on $W\mathcal{A}$, the resulting functorial filtration on $W\mathcal{A}$ is $W_{\leq\lambda}(M,W) = \left(W_{\leq\lambda}M, W_{\leq\lambda}M \cap W\right)$. \\
The situation in (1) is a special case: When the slope filtration on $\mathcal{A}$ is split and hence $\mathcal{A}$ itself abelian semisimple, the equivalence $W\mathcal{A} \to \mathcal{A}$ from \ref{rmk:2}(3) takes the filtration $W$ on $W\mathcal{A}$ to the functorial weight filtration on $\mathcal{A}$.
\end{itemize}
\end{subremark}

\section{Examples}
\namedlabel{sec:4}{Section \thesection}

\subsection{(Abstract) Hodge structures}

Let $\mathcal{B}$ be an abelian category with a rank function $\text{rk}_{\mathcal{B}}$. We continue to assume that $\Lambda$ is a totally ordered, uniquely divisible abelian group.

\begin{subexample}
\namedlabel{ex:4}{Example \thesubsubsection}
Let $\mathcal{B}_1 \subset \text{Fun}\!\left(\Lambda^\circ,\mathcal{B}\right)$ be the full subcategory of descending $\Lambda$-indexed filtrations (\ref{def:6}(1)), in other terms, $\mathcal{B}_1$ is the category of pairs $(M,F)$, where $F = F^{\geq(-)}M$ is a descending $\Lambda$-indexed filtration on $M \in \mathcal{B}$, and whose morphisms are the morphisms in $\mathcal{B}$ respecting the filtrations. The equivalence $\Lambda^\circ \simeq \Lambda$, $\lambda \mapsto -\lambda$ induces an equivalence $\mathcal{B}_1 \simeq W\mathcal{B}^\text{pre}$, which shows that $\mathcal{B}_1$ is a quasi-abelian category (\ref{prop:1}, \ref{ex:8}). As we fix signs in the following, the sense of the filtrations will however be relevant.  \\
Define $\text{rk}(M,F) \defeq \text{rk}_{\mathcal{B}}(M)$ and $\text{deg}(M,F) \defeq \sum_{\nu \in \Lambda} \text{rk}_{\mathcal{B}}\!\left(\text{gr}^\nu_FM\right)\nu$, note that $\text{gr}^\nu_FM$ is nonzero only for finitely many $\nu \in \Lambda$. Then:
\begin{itemize}
\item[(1)] rk is a rank function and $\mu \defeq \text{deg}/\text{rk}$ a $\Lambda$-valued slope function on $\mathcal{B}_1$ which satisfies \eqref{eq:dagger}.
\item[(2)] The corresponding slope filtration is $F_\mu^{\geq\lambda}\!\left(M,F^{\geq(-)}M\right) = \left(F^{\geq\lambda}M,F^{\geq(-)}M \cap F^{\geq\lambda}M\right)$. We say that $F^{\geq(-)}M \cap F^{\geq\lambda}M$ is the filtration induced by $F$ on $F^{\geq\lambda}M$.
\item[(3)] For $\lambda \in \Lambda$, the category $\mathcal{B}_1(\lambda)$ is the category of $\Lambda$-filtered objects of $\mathcal{B}$ with precisely one break of the filtration at $\lambda$.
\item[(4)] $\left(\left(M,F^{\geq(-)}M\right),W_{\leq(-)}\!\left(M,F^{\geq(-)}M\right)\right) \in W\mathcal{B}_1^\text{pre}$ is in $W\mathcal{B}_1$ if and only the canonical morphisms $M \leftarrow \bigoplus_\nu \left(F^{\geq\nu}M \cap W_{\leq\nu}M\right) \to \bigoplus_\nu \text{gr}^W_\nu M$ are isomorphisms, thus giving rise to a \textit{splitting} of the filtration $W$. Here, we also let $W$ denote the filtration $\left(\mathcal{B}_1 \to \mathcal{B}\right) \circ W$ on $M \in \mathcal{B}$. \\
Equivalently, $((M,F),W)$ belongs to $W\mathcal{B}_1$ if and only if $W$ splits $F$, that is if and only if $M \leftarrow \bigoplus_\nu \left(F^{\geq\nu}M \cap W_{\leq\nu}M\right) \to \bigoplus_\nu \text{gr}^\nu_F M$ are isomorphisms.
\item[(5)] Every morphism in $W\mathcal{B}_1$ is strictly compatible with $W$. Its underlying morphism in $\mathcal{B}$ is strictly compatible with both the filtrations $F$ and $W$.
\item[(6)] For all $\nu, \nu' \in \Lambda$, the functors $\text{gr}_\nu^W: W\mathcal{B}_1 \to \mathcal{B}_1$ and $\text{gr}^\nu_F: W\mathcal{B}_1 \to W\mathcal{B}^\text{pre}$ as well as $\text{gr}_\nu^W, \text{gr}^\nu_F: W\mathcal{B}_1 \to \mathcal{B}$ and $\text{gr}^{\nu'}_F\text{gr}_\nu^W \cong \text{gr}_\nu^W\text{gr}^{\nu'}_F: W\mathcal{B}_1 \to \mathcal{B}$ are exact.
\end{itemize}
\end{subexample}

\begin{proof}
Knowing exactness of $\mathcal{B}_1 \hookrightarrow \text{Fun}\!\left(\Lambda^\circ,\mathcal{B}\right)$ when $\mathcal{B}$ is abelian, for assertion (1), one is left to show that if $(M,F) \to (M',F')$ is an epi-monic in $\mathcal{B}_1$ which is not an isomorphism, then $\mu(M,F) < \mu(M',F')$. In this situation, $M \to M'$ is an isomorphism in $\mathcal{B}$, and we may assume without loss of generality that it is the identity of $M$. \\
For a morphism $\iota: (M,F) \to (M,F')$ in $\mathcal{B}_1$ which in $\mathcal{B}$ is the identity, let $\Lambda(\iota)$ be the finite set consisting of all elements $\lambda \in \Lambda$ such that $F^{\geq\lambda}\iota$ is not an isomorphism. We prove by induction on the cardinality $n \geq 0$ of $\Lambda(\iota)$ that either $\iota$ is an isomorphism or $\text{deg}(M,F) < \text{deg}(M,F')$. The induction start is trivial. In the induction step, suppose $n > 0$ and let $\lambda_0$ be the least element of the finite chain $\Lambda(\iota)$. From the filtrations $F$ and $F'$, one builds a filtration $\tilde{F}$ on $M$ with $\tilde{F}^{\geq\lambda}M = F'^{\geq\lambda}M$ when $\lambda \leq \lambda_0$ and $\tilde{F}^{\geq\lambda}M = F^{\geq\lambda}M$ when $\lambda > \lambda_0$. For the resulting morphism $\tilde{\iota}: (M,\tilde{F}) \to (M,F')$, it holds that the set $\Lambda\!\left(\tilde{\iota}\right)$ has cardinality $n-1$, so the induction hypothesis applies. Let $\lambda_1$ be the largest break of the filtration $F$ with the property that $\lambda_1 < \lambda_0$, this exists by definition of $\lambda_0$. Then
\begin{equation*}
\text{deg}(M,\tilde{F}) = \text{deg}(M,F) + \left(\lambda_0-\lambda_1\right)\text{rk}_{\mathcal{A}}\!\left(\frac{F'^{\geq\lambda_0}M}{F^{\geq\lambda_0}M}\right) > \text{deg}(M,F).
\end{equation*}
As $\text{deg}(M,F') \geq \text{deg}(M,\tilde{F})$ by induction, we conclude. \\
Fix $\lambda \in \Lambda$. It is clear that if $0 \neq (M,F) \in \mathcal{B}_1$ is such that the filtration has precisely one break at $\lambda$, then $(M,F)$ is semistable of slope $\lambda$. Conversely, let $(M,F)$ be semistable of slope $\lambda$ and suppose that $\text{gr}^\nu_FM \neq 0$ for some $\nu < \lambda$. Consider the strict subobject $\left(F^{>\nu}M,F'\right) \hookrightarrow (M,F)$ of $M$, where $F'$ is the filtration induced by $F$ on $F^{>\nu}M$. Then $F^{>\nu}M \neq 0$, else $\mu(M,F) < \lambda$, and since addition is compatible with the total order on $\Lambda$, one has
\begin{equation*}
\mu\!\left(F^{>\nu}M,F'\right) = \frac{\text{deg}(M,F) - \sum_{\nu' \leq \nu} \text{rk}_\mathcal{B}\!\left(\text{gr}^{\nu'}_FM\right)\nu'}{\text{rk}_\mathcal{B}(M) - \sum_{\nu' \leq \nu} \text{rk}_\mathcal{B}\!\left(\text{gr}^{\nu'}_FM\right)} > \frac{\lambda\left(\text{rk}_\mathcal{B}(M) - \sum_{\nu' \leq \nu} \text{rk}_\mathcal{B}\!\left(\text{gr}^{\nu'}_FM\right)\right)}{\text{rk}_\mathcal{B}(M) - \sum_{\nu' \leq \nu} \text{rk}_\mathcal{B}\!\left(\text{gr}^{\nu'}_FM\right)} = \lambda,
\end{equation*}
contradicting semistability of $(M,F)$. Hence $\text{gr}^\nu_FM = 0$ must hold for all $\nu < \lambda$, and as then $\lambda\sum_{\nu \geq \lambda} \text{rk}_\mathcal{B}\!\left(\text{gr}^\nu_FM\right) = \text{deg}(M,F) = \sum_{\nu \geq \lambda} \text{rk}_\mathcal{B}\!\left(\text{gr}^\nu_FM\right)\nu$, necessarily also $\text{gr}^\nu_FM = 0$ for all $\nu > \lambda$. The claim (3) about $\mathcal{B}_1(\lambda)$ follows. \\
Now one finds that for all $\lambda \in \Lambda$,
\begin{equation*}
\text{gr}^\lambda_{F_\mu}(M,F) = \left(\text{gr}^\lambda_FM, \frac{F^{\geq(-)}M \cap F^{\geq\lambda}M }{F^{\geq(-)}M \cap F^{>\lambda}M}\right)
\end{equation*}
is an element of $\mathcal{B}_1(\lambda)$. This settles (2), namely that $F_\mu$ is the slope filtration corresponding to the slope function $\mu$. \\
For $\left((M,F),W\right) \in W\mathcal{B}_1^\text{pre}$, it holds that $W_{\leq\lambda}(M,F) = \left(W_{\leq\lambda}M,F^{\geq(-)}M \cap W_{\leq\lambda}M\right)$, therefore $\left((M,F),W\right)$ is in $W\mathcal{B}_1$ if and only if
\begin{equation*}
\frac{F^{\geq\nu'}M \cap W_{\leq\nu}M}{\left(F^{\geq\nu'}M \cap W_{<\nu}M\right) + \left(F^{>\nu'}M \cap W_{\leq\nu}M\right)} = \text{gr}^{\nu'}_F\!\left(\text{gr}^W_\nu(M,F)\right) = \begin{cases} 0 & \nu' \neq \nu \\ \text{gr}^W_\nu M & \nu' = \nu \end{cases}
\end{equation*}
for all $\nu,\nu' \in \Lambda$. By \cite[1.2.6]{ARTICLE:2}, this is in turn equivalent to requiring that the canonical morphisms $M \leftarrow \bigoplus_\nu \left(F^{\geq\nu}M \cap W_{\leq\nu}M\right) \to \bigoplus_\nu \text{gr}^W_\nu M$ be isomorphisms, which proves (4). \\
According to item (0) of \ref{thm:1}, any morphism in $W\mathcal{B}_1$ is strictly compatible with the filtration $W$. In the case at hand, we alternatively see this from \cite[1.2.4.1, 1.2.4.2]{ARTICLE:2}, which moreover shows that any morphism in $\mathcal{B}$ underlying a morphism in $W\mathcal{B}_1$ is strictly compatible also with the filtration $F$. Namely, $\bigoplus_{\rho \geq \nu} \left(F^{\geq\rho}M \cap W_{\leq\rho}M\right) \to F^{\geq\nu}M$ is an isomorphism for $\left((M,F),W\right) \in W\mathcal{B}_1$ and all $\nu \in \Lambda$. Hence we get (5). \\
Finally, we have $\text{gr}^{\nu'}_F\text{gr}^W_\nu = 0$ on $W\mathcal{B}_1$ when $\nu' \neq \nu$, while $\text{gr}^\nu_F\text{gr}^W_\nu \cong \text{gr}^W_\nu$ as functors $W\mathcal{B}_1 \to \mathcal{B}$. All assertions of (6) follow from this. Exactness of $\text{gr}_\nu^W: W\mathcal{B}_1 \to \mathcal{B}_1$ is alternatively item (2) of \ref{thm:1}.
\end{proof}

\begin{intermediate}
In \ref{ex:6} of the appendix, we discuss the case that $\mathcal{B}$ has the structure of a tannakian category.
\end{intermediate}

\begin{subremark}
Let $\mathcal{B}_1$ and $\mu$ on $\mathcal{B}_1$ be as in \ref{ex:4}. Assume in addition that $\mathcal{B}$ is equipped with a slope function $\mu_\mathcal{B}$ for the rank function $\text{rk}_\mathcal{B}$ and let $F_{\mu_\mathcal{B}}$ be the corresponding slope filtration. Then $\mu\!\left(M,F_{\mu_\mathcal{B}}\right) = \mu_\mathcal{B}(M)$ and for $\lambda \in \Lambda$, it holds that $\left(M,F_{\mu_\mathcal{B}}\right) \in \mathcal{B}_1(\lambda)$ if and only if $M \in \mathcal{B}(\lambda)$.
\end{subremark}

\begin{subexample}
\namedlabel{ex:5}{Example \thesubsubsection}
Let $\mathcal{B}_2$ denote the category of triples $(M,F,\bar{F})$, where $F$ and $\bar{F}$ are descending $\Lambda$-indexed filtrations on $M \in \mathcal{B}$, and whose morphisms are the morphisms in $\mathcal{B}$ respecting the filtrations. There is a canonical equivalence between $\mathcal{B}_2$ and the category $\left(\mathcal{B}_1\right)_1$ of $\Lambda$-filtered objects of $\mathcal{B}_1$, so $\mathcal{B}_2$ is quasi-abelian (again \ref{prop:1}, \ref{ex:8}). \\
Define $\text{rk}\!\left(M,F,\bar{F}\right) \defeq \text{rk}_{\mathcal{B}}(M)$ and $\text{deg}\!\left(M,F,\bar{F}\right) \defeq \sum_{\nu \in \Lambda} \text{rk}_{\mathcal{B}}\!\left(\text{gr}^\nu_FM \oplus \text{gr}^\nu_{\bar{F}}M\right)\nu$. Then:
\begin{itemize}
\item[(1)] rk is a rank function and $\mu \defeq \text{deg}/\text{rk}$ a $\Lambda$-valued slope function on $\mathcal{B}_2$ which satisfies \eqref{eq:dagger}. In $\left(\mathcal{B}_1\right)_1$, one computes $\text{gr}^{\bar{\nu}}_{\bar{F}}\text{gr}^\nu_FM$, and it holds that
\begin{equation*}
\text{deg}\!\left(M,F,\bar{F}\right) = \sum_{\bar{\nu},\nu} \text{rk}_{\mathcal{B}}\!\left(\text{gr}^{\bar{\nu}}_{\bar{F}}\text{gr}^\nu_FM\right)\left(\bar{\nu}+\nu\right).
\end{equation*}
\item[(2)] The corresponding slope filtration is $F_\mu^{\geq\lambda}\!\left(M,F,\bar{F}\right) = \left(\sum_{\bar{\nu}+\nu = \lambda} \left(\bar{F}^{\geq\bar{\nu}}M \cap F^{\geq\nu}M\right),F',\bar{F}'\right)$, where $F'$, $\bar{F}'$ are the induced filtrations. 
\item[(3)] For $\lambda \in \Lambda$, $\mathcal{B}_2(\lambda)$ consists of the objects $\left(M,F,\bar{F}\right) \in \mathcal{B}_2$ for which $\text{gr}^{\bar{\nu}}_{\bar{F}}\text{gr}^\nu_FM = 0$ whenever $\bar{\nu} + \nu \neq \lambda$.
\item[(4)] The category $W\mathcal{B}_2$ is the category considered in \cite[Théorème 1.2.10]{ARTICLE:2}.
\item[(5)] Every morphism $f$ in $W\mathcal{B}_2$ is strictly compatible with the filtration $W$. For all $\lambda \in \Lambda$, the morphisms in $\mathcal{B}$ underlying $W_{\leq\lambda}f$ and $\text{gr}^W_\lambda f$ are strictly compatible with the filtrations induced by $F$ and $\bar{F}$.
\item[(6)] For all $\nu,\nu' \in \Lambda$, the functors $\text{gr}^W_\nu: W\mathcal{B}_2 \to \mathcal{B}_2$ and $\text{gr}^{\nu'}_F\text{gr}_\nu^W \cong \text{gr}_\nu^W\text{gr}^{\nu'}_F: W\mathcal{B}_2 \to \mathcal{B}_1$, $\text{gr}^{\nu'}_{\bar{F}}\text{gr}_\nu^W \cong \text{gr}_\nu^W\text{gr}^{\nu'}_{\bar{F}}: W\mathcal{B}_2 \to \mathcal{B}_1$ as well as the functors $\text{gr}_\nu^W, \text{gr}^\nu_F, \text{gr}^\nu_{\bar{F}}: W\mathcal{B}_2 \to \mathcal{B}$ and
$\text{gr}_{\bar{F}}\text{gr}^{\nu'}_F\text{gr}_\nu^W \cong \text{gr}_F\text{gr}^{\nu-\nu'}_{\bar{F}}\text{gr}_\nu^W: W\mathcal{B}_2 \to \mathcal{B}$ are exact. Here, $\text{gr}_{\bar{F}} = \bigoplus_{\bar{\nu}} \text{gr}^{\bar{\nu}}_{\bar{F}}$.
\end{itemize}
\end{subexample}

\begin{subremark}
The example given in \ref{rmk:4}(1) shows that the functors $\text{gr}^\nu_F, \text{gr}^\nu_{\bar{F}}: \mathcal{B}_2 \to \mathcal{B}_1$ and $\text{gr}^{\bar{\nu}}_{\bar{F}}\text{gr}^\nu_F \cong \text{gr}^\nu_F\text{gr}^{\bar{\nu}}_{\bar{F}}: \mathcal{B}_2 \to \mathcal{B}$ are not exact in general, and that $F^{\geq\lambda}_\mu$ on $\mathcal{B}_2$ is not exact.
\end{subremark}

\begin{subdefinition}
\namedlabel{def:4}{Definition \thesubsubsection}
In line with \ref{def:1} and motivated by \cite[Théorème 1.2.10]{ARTICLE:2}, we say that the objects of $W\mathcal{B}_2$ are \textit{abstract mixed Hodge structures}.
\end{subdefinition}

\begin{proof}[Proof of the claims of \ref{ex:5}]
Let $f: (M,F,\bar{F}) \to (M',F',\bar{F}')$ be an epi-monic in $\mathcal{B}_2$. As in \ref{ex:4}, $f: M \to M'$ is an isomorphism in $\mathcal{B}$, and for the purpose of the proof of (1), we may again assume that it is the identity of $M$. In $\mathcal{B}_2$, we can thus factor $f$ into epi-monics $(M,F,\bar{F}) \to (M,F',\bar{F}) \to (M,F',\bar{F}')$. Applying the argument from the proof of \ref{ex:4}(1) to both morphisms in the factorisation, we get that $\mu$ is a $\Lambda$-valued slope function on $\mathcal{B}_2$ which satisfies \eqref{eq:dagger}. \\
Furthermore, one computes
\begin{equation*}
\begin{split}
\sum_{\nu \in \Lambda} \text{rk}_{\mathcal{B}}\!\left(\text{gr}^\nu_FM \oplus \text{gr}^\nu_{\bar{F}}M\right)\nu &= \left(\sum_{\bar{\nu} \in \Lambda} \text{rk}_{\mathcal{B}}\!\left(\text{gr}^{\bar{\nu}}_{\bar{F}}M\right)\bar{\nu}\right) + \left(\sum_{\nu \in \Lambda} \text{rk}_{\mathcal{B}}\!\left(\text{gr}^\nu_FM\right)\nu\right) \\
&= \sum_{\bar{\nu},\nu} \left(\text{rk}_{\mathcal{B}}\!\left(\text{gr}^\nu_F\text{gr}^{\bar{\nu}}_{\bar{F}}M\right)\bar{\nu} + \text{rk}_{\mathcal{B}}\!\left(\text{gr}^{\bar{\nu}}_{\bar{F}}\text{gr}^\nu_FM\right)\nu\right) \\
&= \sum_{\bar{\nu},\nu} \text{rk}_{\mathcal{B}}\!\left(\text{gr}^{\bar{\nu}}_{\bar{F}}\text{gr}^\nu_FM\right)\left(\bar{\nu}+\nu\right),
\end{split}
\end{equation*}
which completes the demonstration of (1). \\
By this formula, any $0 \neq (M,F,\bar{F}) \in \mathcal{B}_2$ with the property that $\text{gr}^{\bar{\nu}}_{\bar{F}}\text{gr}^\nu_FM = 0$ whenever $\bar{\nu} + \nu \neq \lambda$ has slope $\lambda$. For such a triple, it holds that $M \cong \bigoplus_{\bar{\nu}+\nu = \lambda} \text{gr}^{\bar{\nu}}_{\bar{F}}\text{gr}^\nu_FM$  in $\mathcal{B}$ according to \cite[1.2.6]{ARTICLE:2}. As this is also an isomorphism in $\mathcal{B}_2$ and as each summand is semistable, we deduce \cite[Lemma 1.3.7(7)]{ARTICLE:1} that also $(M,F,\bar{F})$ is semistable. \\
Conversely, let $(M,F,\bar{F})$ be semistable of slope $\lambda$. As in \ref{ex:4}, in order to see that $\text{gr}^{\bar{\nu}}_{\bar{F}}\text{gr}^\nu_FM = 0$ holds whenever $\bar{\nu} + \nu \neq \lambda$, it suffices to show this for all $\bar{\nu} + \nu > \lambda$. Suppose that $\text{gr}^{\bar{\nu}'}_{\bar{F}}\text{gr}^{\nu'}_FM \neq 0$ for some $\bar{\nu}' + \nu' > \lambda$, then $M' \defeq \bar{F}^{\geq\bar{\nu}'}M \cap F^{\geq\nu'}M$ is nonzero, $M'$ with the induced filtrations $F'$, $\bar{F}'$ is a strict subobject of $(M,F,\bar{F})$ and $\mu(M',F',\bar{F}') > \lambda$, a contradiction. This establishes the description of $\mathcal{B}_2(\lambda)$ in (3). \\
Now it follows from \cite[Lemma 1.5.1(ii)]{ARTICLE:3} that $F_\mu$ as in (2) is the slope filtration corresponding to the slope function $\mu$. It is also an immediate consequence that $W\mathcal{B}_2$ is the category considered in \cite[Théorème 1.2.10]{ARTICLE:2}, as claimed in (4). \\
Assertions (5) and (6) are part of \cite[Théorème 1.2.10]{ARTICLE:2} and can also be derived from \ref{thm:1}. To illustrate this for the nontrivial part of (5), it is enough to observe that \ref{thm:1} implies that $W_{\leq\lambda}f$ and $\text{gr}^W_\lambda f$ are strict in $\mathcal{B}_2$.
\end{proof}

\begin{subvariant}
\namedlabel{var:1}{Variant \thesubsubsection}
Let $G: \mathcal{A} \to \mathcal{B}$ be an exact faithful functor of abelian categories (hence $G$ detects exactness), where $\mathcal{B}$ is equipped with a rank function. Let $W^G\mathcal{B}_2$ be the category whose objects are $(X,W) \in W\mathcal{A}^\text{pre}$ together with two descending filtrations $F, \bar{F}$ on $GX$ such that $(GX,F,\bar{F},GW) \in W\mathcal{B}_2$, and whose morphisms are morphisms $f$ in $\mathcal{A}$ such that $f$ respects the filtration $W$ and $Gf$ respects the filtrations $F$ and $\bar{F}$. Then $W^G\mathcal{B}_2$ is an abelian category, any morphism $f$ in $W^G\mathcal{B}_2$ has the property that $f$ in $\mathcal{A}$ is strictly compatible with the filtration $W$ and $\text{gr}^W$ is an exact functor $W^G\mathcal{B}_2 \to \mathcal{A}$. \\
Applied with $\mathcal{A}$ the category of finite-dimensional $\mathbb{Q}$-vector spaces, $\mathcal{B}$ the category of finite-dimensional $\mathbb{C}$-vector spaces and $G$ the scalar extension functor, this yields Deligne's theorem \cite[Théorème 2.3.5]{ARTICLE:2} on the abelianess of the category of mixed Hodge structures in the original sense (that is with rational lattice, \cite[Définition 2.3.1]{ARTICLE:2}).
\end{subvariant}

\subsection{Twistor structures}

Let $X \defeq \mathbb{P}^1_\mathbb{C}$ be the projective line, regarded either as scheme or, through its analytification, as complex analytic space. In either case, let $\mathcal{O}_X$ denote the structure sheaf. \\
When adopting the first perspective, the complex numbers may be replaced by any field. Then \ref{rmk:5}(1) applies as long as the characteristic is zero. The assertions of \ref{ex:3} remain valid without the assumption on the characteristic.

\begin{subexample}
\namedlabel{ex:3}{Example \thesubsubsection}
A \textit{twistor structure} \cite{MISC:3} is a vector bundle, that is locally free $\mathcal{O}_X$-module of finite (necessarily constant) rank, on $X$. By a standard result of Grothendieck, every vector bundle $\mathcal{V}$ on $X$ is a direct sum of line bundles, hence of the form
\begin{equation*}
\mathcal{V} \cong \bigoplus_{i = 1}^{\text{rk}(\mathcal{V})} \mathcal{O}_X\!\left(n_i\right)
\end{equation*}
for suitable $n_i \in \mathbb{Z}$. Its slope in the classical sense is
\begin{equation*}
\mu(\mathcal{V}) = \frac{\sum_{i = 1}^{\text{rk}(\mathcal{V})} n_i}{\text{rk}(\mathcal{V})},
\end{equation*}
which indeed defines a slope function on the quasi-abelian category $\mathcal{A}$ of twistor structures, see for instance \ref{rmk:5}(1) (elaborating on \cite[3.1.1]{ARTICLE:1}). It is clear that with respect to this slope function, the twistor structure $\mathcal{V}$ is semistable of slope $n$ if and only if $n \in \mathbb{Z}$ and $\mathcal{V}$ is a direct sum of $\text{rk}(\mathcal{V})$ many copies of $\mathcal{O}_X(n)$. \\
The category $W\mathcal{A}$ is the category of \textit{mixed twistor structures} which was introduced and studied in \cite{MISC:3}.
\end{subexample}

\begin{subremark}
\namedlabel{rmk:5}{Remark \thesubsubsection}
$\text{}$
\begin{itemize}
\item[(1)] Since $X$ is in particular a normal curve (a curve for us is a one-dimensional integral scheme separated and of finite type over a field), the category $\mathcal{A}$ of vector bundles on $X$ is the quasi-abelian category of torsionfree coherent sheaves on $X$ with left abelian envelope $\mathcal{L}_\mathcal{A}$ the category of all coherent sheaves on $X$. \\
More specifically for $X = \mathbb{P}^1_\mathbb{C}$, $\mathcal{A}$ can be given the structure of a quasi-tannakian category (\ref{app:A.3}) from \cite[Examples 2.1.4(2)]{ARTICLE:1}. The quasi-tannakian determinant of $\mathcal{V} \cong \bigoplus_{i = 1}^{\text{rk}(\mathcal{V})} \mathcal{O}_X\!\left(n_i\right)$ is $\text{det}(\mathcal{V}) = \bigwedge^{\text{rk}(\mathcal{V})}\mathcal{V} \cong \mathcal{O}_X\!\left(\sum_i n_i\right)$, and as $\delta: \text{Pic}(\mathcal{A}) \to \mathbb{Z}$, $\mathcal{O}_X(n) \mapsto n$ is a homomorphism satisfying $\delta\!\left(\mathcal{O}_X(n)\right) \geq 0$ if there is a nonzero morphism $\mathcal{O}_X \to \mathcal{O}_X(n)$, with equality only if this is an isomorphism, the function $\mu = (\delta \circ \text{det})/\text{rk}$ is a $\mathbb{Q}$-valued slope function on $\mathcal{A}$ (\cite[Theorem 2.2.8]{ARTICLE:1}, \ref{app:A.3}).
\item[(2)] When in \ref{def:4}, $\mathcal{B}$ is the category of finite-dimensional complex vector spaces, the category $W\mathcal{B}_2$ of abstract mixed Hodge structures maps faithfully to the category of mixed twistor structures via the Rees construction \cite[Proposition 1.2]{MISC:3}. If instead of the category $W\mathcal{B}_2$, one considers the category of mixed Hodge structures in its original sense (\ref{var:1}), the corresponding result is \cite[Lemma 1.1]{MISC:3}.
\end{itemize}
\end{subremark}

\subsection{Monodromy-weight structures}

\begin{subnotation}
Let $K$ be a local field (a finite extension of $\mathbb{Q}_p$ or $\mathbb{F}_p((t))$), $l \neq p$ a prime number different from the residue characteristic of $K$ and $\overline{K}$ a separable closure of $K$. We let $\mathcal{A}$ denote the abelian category of $l$-adic representations of the absolute Galois group $G_K = \text{Gal}\!\left(\overline{K}/K\right)$ in the sense of \cite[1.1.6]{ARTICLE:5}, that is the filtered colimit (with exact transition functors) of the categories of finite-dimensional, continuous $G_K$-representations over finite extensions of $\mathbb{Q}_l$. Every object of $\mathcal{A}$ has a finite-dimensional $\overline{\mathbb{Q}_l}$-vector space $V$ attached to it and we will often just write $V \in \mathcal{A}$. We fix a field isomorphism $\iota: \overline{\mathbb{Q}_l} \cong \mathbb{C}$, with respect to which the weights will be computed, and, to simplify notation, an orientation $\mathbb{Z}_l(1) \cong \mathbb{Z}_l$ of $\overline{K}$. \\
Given an $l$-adic representation $\rho: G_K \to \text{Aut}(V)$, write $N$ for the corresponding nilpotent monodromy operator \cite[1.7.2, 1.7.3]{ARTICLE:5} on $V$. By construction of $N$, there is an open normal subgroup $H$ of the inertia group $I_K \subset G_K$ such that $\rho(h) = \text{exp}\!\left(t_l(h)N\right)$ holds for all $h \in H$, where $t_l: I_K \to \mathbb{Z}_l$ denotes the tame $l$-adic character, associated to the action of $I_K$ on the $l^n$th roots of a uniformiser of $K$ (\cite[Theorem 2.26]{MISC:4}). For $h \in I_K$ and any element $g$ of the Weil group $W_K \cong I_K \rtimes \langle F \rangle$, it verifies $t_l\!\left(ghg^{-1}\right) = q^{a(g)}t_l(h)$. Here, $F$ is any geometric Frobenius element, $q$ denotes the cardinality of the residue field $k$ of $K$ and $a$ is the homomorphism
\[\begin{tikzcd}[row sep=small]
{W_K} && {\mathbb{Z}} \\
{G_K} & {\text{Gal}\!\left(\overline{k}/k\right)} & {\widehat{\mathbb{Z}}} \\
F^{-1} & {\left(x \mapsto x^q\right)} & 1
\arrow["a", from=1-1, to=1-3]
\arrow[hook, from=1-1, to=2-1]
\arrow[hook, from=1-3, to=2-3]
\arrow[two heads, from=2-1, to=2-2]
\arrow["\cong", from=2-2, to=2-3]
\arrow[maps to, from=3-1, to=3-2]
\arrow[maps to, from=3-2, to=3-3]
\end{tikzcd}\]
so that in terms of class field theory, $q^{a(g)}$ is the normalised absolute value of the element of $K^\times$ corresponding to $\overline{g} \in W_K^\text{ab}$ under the reciprocity isomorphism. \\
When $h \in H$, we consequently have $\rho\!\left(ghg^{-1}\right) = \rho(h)^{q^{a(g)}}$. It follows that $\rho(g)N = q^{a(g)}N\rho(g)$ holds for all $g \in W_K$, and in particular,
\begin{equation}
\label{eq:*}
\tag{$*$}
\rho(F)N = q^{-1}N\rho(F)
\end{equation}
for any geometric Frobenius lift $F \in W_K$. Remembering twists, this means that $N$ is a morphism of representations $V(1) \to V$.
\end{subnotation}

\begin{subdefinition}
Fix the rank function on $\mathcal{A}$ which computes the dimension of the underlying vector space.
\begin{itemize}
\item[(1)] If $\overline{W}$ denotes the Frobenius-weight filtration \cite[Proposition-définition 1.7.5]{ARTICLE:5}, we define the \textit{monodromy slope} of $V \in \mathcal{A}$ as
\begin{equation*}
\mu(V) \defeq \frac{1}{\text{dim}(V)}\left(\sum_{\nu \in \mathbb{R}} \text{dim}\!\left(\text{gr}^{\overline{W}}_\nu V\right)\nu\right).
\end{equation*}
\end{itemize}
We use different notation than in \cite{ARTICLE:5} in order to distinguish between various filtrations. The Frobenius-weight filtration $\overline{W}$ on $V$ is the unique ascending $\mathbb{R}$-indexed filtration in $\mathcal{A}$ with the property that for all $\nu \in \mathbb{R}$, the Frobenius action on $\text{gr}^{\overline{W}}_\nu V$ is $\iota$-pure of weight $\nu$, which means that images of the eigenvalues of $F$ under $\iota$ have absolute value $q^\frac{\nu}{2}$. This does not depend on the choice of Frobenius lift \cite[Lemme 1.7.4]{ARTICLE:5}. Concretely, $\overline{W}_{\!\leq\nu}V$ is the sum of generalised Frobenius eigenspaces of $V$ of weight $\leq \nu$, hence satisfies $N\overline{W}_{\!\leq\nu} V \subset \overline{W}_{\!\leq\nu-2}V$ by \eqref{eq:*}. $\overline{W}$ is a functorial filtration on $\mathcal{A}$ and it is exact due to triangularisability, so that $\mu$ defined above is indeed a $\mathbb{R}$-valued slope function on $\mathcal{A}$. Note that condition \eqref{eq:dagger} is empty in this case because $\mathcal{A}$ is abelian. \\
$\iota$ being fixed, we will just say pure instead of $\iota$-pure in the following.
\begin{itemize}
\item[(2)] We call $V$ a \textit{pure monodromy-weight structure of weight $\lambda \in \mathbb{R}$} if $\text{gr}^{\overline{W}}_\nu V = 0$ holds whenever $\nu \notin \lambda + \mathbb{Z}$ and if $N^a: \text{gr}^{\overline{W}}_{\lambda + a} V \to \text{gr}^{\overline{W}}_{\lambda - a} V$ is a vector space isomorphism for all $a \in \mathbb{N}$.
\end{itemize}
For such $V$, $N^a: (\text{gr}^{\overline{W}}_{\lambda + a} V)(a) \to \text{gr}^{\overline{W}}_{\lambda - a} V$ are thus isomorphisms in $\mathcal{A}$. \\
Item (2) can be rephrased as follows: Let $M$ denote the monodromy filtration on $V$ attached to $N$ \cite[Proposition 1.6.1, 1.7.2, 1.7.3]{ARTICLE:5}, that is the unique ascending $\mathbb{Z}$-indexed filtration by subspaces such that $NM_aV \subset M_{a-2}V$ and such that $N^a$ induces an isomorphism $\text{gr}^M_aV \to \text{gr}^M_{-a}V$ for all $a \in \mathbb{N}$. Then $V$ is a pure monodromy-weight structure of weight $\lambda$ if and only if the Frobenius-weight filtration $\overline{W}$ equals the monodromy filtration $M$ shifted by $\lambda$. \\
The monodromy filtration is in fact a filtration in $\mathcal{A}$, because uniqueness implies that $M_aV \subset V$ are $G_K$-invariant. It is functorial on $\mathcal{A}$, as $f \circ N' = N \circ f$ holds for morphisms $f: V' \to V$ in $\mathcal{A}$, but it fails to be exact.
\end{subdefinition}

\begin{subproposition}
\namedlabel{prop:3}{Proposition \thesubsubsection}
A nonzero object $V$ of $\mathcal{A}$ is semistable of monodromy slope $\lambda$ if and only if it is a pure monodromy-weight structure of weight $\lambda$.
\end{subproposition}

\begin{intermediate}
As a consequence, for any $\lambda \in \mathbb{R}$, the full subcategory of $\mathcal{A}$ consisting of zero and the pure monodromy-weight structures of weight $\lambda$ is abelian. \\

We will give two proofs of \ref{prop:3}. The first one is constructive, in that for an arbitrary representation $V \in \mathcal{A}$, it exhibits the universal destabilising subobject with respect to the monodromy slope function (\ref{cor:3}).
\end{intermediate}

\begin{proof}[First proof of \ref{prop:3}]
$\text{}$ \\
($\impliedby$) Let $V$ be a pure monodromy-weight structure of weight $\lambda$ and compute
\begin{equation*}
\mu(V) = \frac{1}{\text{dim}(V)}\left(\sum_{\nu \in \mathbb{R}} \text{dim}\!\left(\text{gr}^{\overline{W}}_\nu V\right)\nu\right) = \frac{1}{\text{dim}(V)}\lambda\left(\sum_{a \in \mathbb{Z}} \text{dim}\!\left(\text{gr}^{\overline{W}}_{\lambda + a} V\right)\right) = \lambda.
\end{equation*}
Let $V' \hookrightarrow V$ be a subobject in $\mathcal{A}$. Because of exactness of the Frobenius-weight filtration $\overline{W}$, $\text{gr}^{\overline{W}}_\nu V' \hookrightarrow \text{gr}^{\overline{W}}_\nu V$ is monic for all $\nu \in \mathbb{R}$, and for every $a \in \mathbb{N}$, we have a commutative diagram
\[\begin{tikzcd}
{\text{gr}^{\overline{W}}_{\lambda + a} V} & {\text{gr}^{\overline{W}}_{\lambda - a} V} \\
{\text{gr}^{\overline{W}}_{\lambda + a} V'} & {\text{gr}^{\overline{W}}_{\lambda - a} V'}
\arrow["\cong"', "{N^a}", from=1-1, to=1-2]
\arrow[hook', from=2-1, to=1-1]
\arrow["{(N')^a}", from=2-1, to=2-2]
\arrow[hook', from=2-2, to=1-2]
\end{tikzcd}\]
showing that $(N')^a: \text{gr}^{\overline{W}}_{\lambda + a} V' \hookrightarrow \text{gr}^{\overline{W}}_{\lambda - a} V'$ is monic. It follows that $\sum_{a \in \mathbb{Z}} \text{dim}\!\left(\text{gr}^{\overline{W}}_{\lambda + a} V'\right)a \leq 0$ and consequently
\begin{equation*}
\mu\!\left(V'\right) = \frac{1}{\text{dim}\!\left(V'\right)}\left(\text{dim}\!\left(V'\right)\lambda + \sum_{a \in \mathbb{Z}} \text{dim}\!\left(\text{gr}^{\overline{W}}_{\lambda + a} V'\right)a\right) \leq \lambda.
\end{equation*}
Hence $V$ is semistable of monodromy slope $\lambda$. \\
($\implies$) We will prove the converse by induction on the rank $\geq 1$ of $V$, the induction start being trivial. For the purpose of the induction step, we will construct a nonzero pure monodromy-weight structure of weight $\lambda$ inside $V$. In fact, for an arbitrary representation $V \in \mathcal{A}$, what we will construct the largest, in terms of weight and size, pure monodromy-weight structure inside $V$ (as explained in the proof of \ref{cor:3}). \\
Fix $0 \neq V \in \mathcal{A}$. Let $F$, $F''$ be two Frobenius lifts and let $\rho(F)$, $\rho\!\left(F''\right)$ denote the automorphisms through which they act on $V$. For $\eta \in \mathbb{R}$, write $V_\eta$ and $V''_\eta$ for the sum of generalised eigenspaces of weight $\eta$ of $\rho(F)$ and $\rho\!\left(F''\right)$ respectively. Since $\left\{\overline{\rho\!\left(F''\right)^m\rho(F)^{-m}} \in \rho\!\left(I_K\right)/\rho(H) \mid m \geq 1\right\}$ is finite, it holds that $\rho\!\left(F''\right)^n = \text{exp}(zN)\rho(F)^n = \text{exp}(wN)\rho(F)^n\text{exp}(wN)^{-1}$ for suitable $n \geq 1$, $z \in \mathbb{Z}_l$ and $w = z(1-q^{-n})^{-1}$, compare the proof of \cite[1.7.5]{ARTICLE:5}. Put $U_\eta \defeq \bigoplus_{a \in \mathbb{Z}} \left(V_{\eta+a} \cap M_aV\right)$ (recall that $M$ denotes the monodromy filtration on $V$) and define $U''_\eta$ analogously. Then $\text{exp}(wN)$ maps $U_\eta$ into $U''_\eta$ and at the same time fixes $U_\eta$. Therefore, $U_\eta = U''_\eta$ is independent of the choice of Frobenius lift and in particular invariant under the action of $W_K$, thus already a $G_K$-invariant subspace of $V$. \\
The isomorphism $N^a: \text{gr}^M_aV(a) \stackrel{\cong}{\to} \text{gr}^M_{-a}V$ for $a \geq 0$ yields for every $b \in \mathbb{Z}$ an isomorphism
\begin{equation*}
N^a: \frac{V_{\eta+b+2a} \cap M_aV}{V_{\eta+b+2a} \cap M_{a-1}V} =\left(\text{gr}^M_aV(a)\right)_{\eta+b} \stackrel{\cong}{\to} \left(\text{gr}^M_{-a}V\right)_{\eta+b} = \frac{V_{\eta+b} \cap M_{-a}V}{V_{\eta+b} \cap M_{-a-1}V},
\end{equation*}
so that the monodromy filtration on $U_\eta$ is given by $M_aU_\eta = \bigoplus_{b \in \mathbb{Z}} \left(V_{\eta+b} \cap M_{\min\{a,b\}}V\right)$ with
\begin{equation*}
\text{gr}^M_aU_\eta = \bigoplus_{b \geq a} \frac{V_{\eta+b} \cap M_aV}{V_{\eta+b} \cap M_{a-1}V} = \bigoplus_{b \geq a} \left(\text{gr}^M_aV\right)_{\eta+b}.
\end{equation*}
Take $\nu \defeq \text{max}\!\left\{\eta \in \mathbb{R} \mid U_\eta \neq 0\right\}$, see \ref{rmk:8}, and put $U \defeq U_\nu$. Then for all $a \in \mathbb{Z}$, the weights on $\text{gr}^M_aV$ are bounded above by $\nu + a$. Hence $\text{gr}^M_aU = \left(\text{gr}^M_aV\right)_{\nu+a}$ and so the Frobenius action on $\text{gr}^M_aU$ is pure of weight $\nu + a$, which means that $U \neq 0$ is a pure monodromy-weight structure of weight $\nu$. Moreover, using exactness of $\overline{W}$,
\begin{equation*}
\begin{split}
\mu(V) &= \frac{1}{\text{dim}(V)}\!\left(\sum_{\eta \in \mathbb{R}}\text{dim}\!\left(\text{gr}^{\overline{W}}_\eta V\right)\eta - \sum_{a \in \mathbb{Z}}\text{dim}\!\left(\text{gr}^M_aV\right)a\right) \\
&= \frac{1}{\text{dim}(V)}\sum_{\eta \in \mathbb{R}, a \in \mathbb{Z}}\text{dim}\!\left(\text{gr}^{\overline{W}}_\eta\text{gr}^M_aV\right)(\eta-a) \\
&\leq \nu\frac{1}{\text{dim}(V)}\sum_{\eta \in \mathbb{R}, a \in \mathbb{Z}}\text{dim}\!\left(\text{gr}^{\overline{W}}_\eta\text{gr}^M_aV\right) = \nu.
\end{split}
\end{equation*} 
Suppose now that $V$ is semistable of slope $\lambda$. Then we must have $\nu = \lambda$. If $U = V$, we are done with the proof of the proposition. If $U \neq V$, the representation $V/U$ is again semistable of slope $\lambda$ and of smaller rank than $V$, thus a pure monodromy-weight structure of weight $\lambda$ by the induction hypothesis. Since pure monodromy-weight structures of fixed weight satisfy a two-out-of-three property in $\mathcal{A}$, we deduce that, in any case, $V$ is a pure monodromy-weight structure of weight $\lambda$.
\end{proof}

\begin{intermediate}
Actually, the case $U \neq V$ does not occur when $V$ is semistable, as the following corollary shows.
\end{intermediate}

\begin{subcorollary}
\namedlabel{cor:3}{Corollary \thesubsubsection}
We keep the notation from the first proof of \ref{prop:3}. Let $\nu$ be the largest weight such that $U \defeq \bigoplus_{a \in \mathbb{Z}} \left(V_{\nu+a} \cap M_aV\right)$ is nonzero. Then $U \hookrightarrow V$ is the universal destabilising subobject of $V$.
\end{subcorollary}

\begin{proof}
From the proof of \ref{prop:3}, we know that $U$ is semistable of slope $\nu$. \\
While not strictly, every morphism in $\mathcal{A}$ respects the monodromy filtration, and in particular, $M_aV' \subset M_aV$ holds for a subrepresentation $V' \subset V$. So for any pure monodromy-weight structure $V'$ in $V$ of weight $\eta = \mu\!\left(V'\right)$, we have $V'_{\eta+a} = V'_{\eta+a} \cap M_aV' \subset V_{\eta+a} \cap M_aV$, and consequently, $V' \hookrightarrow V$ factors through $U_\eta = \bigoplus_{a \in \mathbb{Z}} \left(V_{\eta+a} \cap M_aV\right) \hookrightarrow V$. This shows both that $\mu\!\left(V'\right) \leq \nu$ and that $V' \hookrightarrow V$ factors through $U = U_\nu \hookrightarrow V$ whenever $\mu\!\left(V'\right) = \nu$. Since by \ref{prop:3}, the pure monodromy weight-structures in $V$ are precisely the semistable subobjects, we conclude that $U$ is the universal destabilising subobject.
\end{proof}

\begin{subremark}
\namedlabel{rmk:8}{Remark \thesubsubsection}
Concretely, one has $\nu = \text{max}\!\left\{\nu_0, \nu_1-1\right\}$, where $\nu_0$ is the largest weight such that $V_{\nu_0} \cap M_0V \neq 0$ and $\nu_1$ is the largest weight such that $V_{\nu_1} \cap M_1V \neq 0$. This is because for every $a \in \mathbb{Z}$, $\text{gr}^M_aV$ is a subobject or a quotient of the appropriate twist of $\text{gr}^M_0V$ or $\text{gr}^M_1V$. Note also that when $V \neq 0$, at least one of $\text{gr}^M_0V$ or $\text{gr}^M_1V$ is nonzero.
\end{subremark}

\begin{intermediate}
Let us give another proof of the implication ($\implies$) of \ref{prop:3} built on an argument due to Deligne \cite[Proof of Théorème 1.8.4]{ARTICLE:5}. It uses that with respect to the the canonical tannakian structure on $\mathcal{A}$, the monodromy slope is $\otimes$-multiplicative (see \ref{app:A.3}). \\
To see this, observe that the obvious filtrations on a tensor product $V \otimes V'$ in $\mathcal{A}$ induced by the monodromy filtrations and the Frobenius-weight filtrations of $V$, $V'$ have the characterising properties of the monodromy filtration and the Frobenius-weight filtration of $V \otimes V'$ respectively (compare \cite[Proposition 1.6.9]{ARTICLE:5} for the monodromy filtration). \\
As an aside, the monodromy slope is in particular determinantal (again see \ref{app:A.3}), so $\text{dim}(V)\mu(V)$ is the $\iota$-weight of the determinant of $V$ in $\mathcal{A}$.
\end{intermediate}

\begin{proof}[Second proof of \ref{prop:3}]
Let $P_{-j}$ denote the $j$th primitive part, that is the kernel of $N: \text{gr}^M_{-j}V \to \text{gr}^M_{-j-2}V$ ($P_{-j} = 0$ when $j < 0$), and continue to write $P_{-j}(m)$ for its $m$th Tate twist, shifting weights by $-2m$. It holds that \cite[1.6.14, sign corrected]{ARTICLE:5}
\begin{equation*}
\text{gr}^M_iV \cong \bigoplus_{\substack{j \geq |i| \\ j \equiv i (2)}} P_{-j}\!\left(-\frac{i+j}{2}\right)
\end{equation*}
and
\begin{equation*}
P_{-j} \cong \text{gr}^{\tilde{M}}_{-j}\!\left(\text{Ker}N\right)
\end{equation*}
as representations, where $\tilde{M}$ is the filtration induced by $M$ on $\text{Ker}N$ (not the monodromy filtration on $\text{Ker}N$, which is trivial). Moreover, the formula
\begin{equation*}
P_{-j}\!\left(V^\lor\right) = \left(P_{-j}(-j)\right)^\lor
\end{equation*}
holds for the primitive parts of the dual representation, and $P_{-j} \otimes P_{-j}(-j)$ is a direct summand of $P_0\!\left(V \otimes V\right)$. \\
If $V$ is semistable of slope $\lambda$, in order to verify that $V$ is a pure monodromy-weight structure of weight $\lambda$, it suffices to show that the Frobenius action on $P_{-j}$ is pure of weight $\lambda - j$. To this end, we first observe that the weights of the Frobenius action on $\text{Ker}N$ are $\leq \lambda$. Assuming by contradition that the Frobenius acts on $\text{Ker}N$ with a weight $\nu > \lambda$, the sum of generalised Frobenius eigenspaces of $\text{Ker}N$ of weight $\nu$, $\left(\text{Ker}N\right)_\nu$, is $G_K$-invariant because $N = 0$ on $\text{Ker}N$. But $\left(\text{Ker}N\right)_\nu$ has monodromy slope $\nu > \lambda$, a contradiction. Therefore, the weights of the Frobenius action on the primitive parts are $\leq \lambda$. \\
Now let $\alpha$ be a Frobenius eigenvalue on $P_{-j}$. We need to show that its weight is $w(\alpha) = \lambda - j$. On the one hand, $\alpha^2q^j$ is a Frobenius eigenvalue on $P_{-j} \otimes P_{-j}(-j)$, hence on $P_0\!\left(V \otimes V\right)$, and since $V \otimes V$ is semistable of slope $2\lambda$ due to $\otimes$-multiplicativity of the monodromy slope, we have
\begin{equation*}
2w(\alpha) +2j = w\!\left(\alpha^2q^j\right) \leq 2\lambda
\end{equation*}
according to the preliminary observation, that is $w(\alpha) \leq \lambda - j$. On the other hand, $\alpha^{-1}q^{-j}$ is a Frobenius eigenvalue on $\left(P_{-j}(-j)\right)^\lor$, thus on $P_{-j}\!\left(V^\lor\right)$. As $V^\lor$ is semistable of slope $-\lambda$ (use \cite[Theorem 2.3.3(2b)]{ARTICLE:1}), we can apply the previous result to get
\begin{equation*}
-w(\alpha) - 2j = w\!\left(\alpha^{-1}q^{-j}\right) \leq -\lambda - j,
\end{equation*}
so $w(\alpha) \geq \lambda - j$. In summary, $w(\alpha) = \lambda - j$ as required.
\end{proof}

\begin{subdefinition}
\namedlabel{def:7}{Definition \thesubsubsection}
Let $(V,W) \in W\mathcal{A}^\text{pre}$. We call $(V,W)$ a \textit{mixed monodromy-weight structure} if for every $\lambda \in \mathbb{R}$, $\text{gr}^{\overline{W}}_\nu\text{gr}^W_\lambda V = 0$ holds whenever $\nu \notin \lambda + \mathbb{Z}$ and $N^a: \text{gr}^{\overline{W}}_{\lambda + a}\text{gr}^W_\lambda V \to \text{gr}^{\overline{W}}_{\lambda - a}\text{gr}^W_\lambda V$ is a vector space isomorphism for all $a \in \mathbb{N}$.
\end{subdefinition}

\begin{intermediate}
It amounts to the same to take for $\overline{W}$ and $N$ the Frobenius-weight filtration and nilpotent monodromy operator of $\text{gr}^W_\lambda V \in \mathcal{A}$ or those induced from $V$. \\

By \ref{prop:3}, the full subcategory of $W\mathcal{A}^\text{pre}$ consisting of the mixed-monodromy weight structures is the category $W\mathcal{A}$, hence abelian.
\end{intermediate}

\begin{subremark}
\namedlabel{rmk:3}{Remark \thesubsubsection}
$\text{}$
\begin{itemize}
\item[(1)] The Frobenius-weight filtration $\overline{W}$ is a functorial weight filtration on $\mathcal{A}$. When $N = 0$ on $V \in \mathcal{A}$, $\overline{W}$ splits on $V$. \\
Failure of exactness of the monodromy-slope filtration is equivalent to the existence of pure monodromy-weight structures with nontrivial monodromy $N$. These structures give rise to weight filtrations other than $\overline{W}$.
\item[(2)] Mixed monodromy-weight structures appear without name in \cite[1.8]{ARTICLE:5} and have been taken up in unpublished work by Scholl \cite{MISC:5}. \\
A variant is Fontaine-Ouyang's category of geometric weighted representations, which is asserted (without proof) to be abelian in \cite[Theorem 2.49]{MISC:4}. They are related to mixed monodromy-weight structures as follows: Recall the notation from the second proof of \ref{prop:3}. When $V \in \mathcal{A}$ is a $\iota$-pure monodromy-weight structure of weight $\lambda$, the Weil-Deligne representation attached to $V$ is isomorphic to the direct sum of Weil-Deligne representations $\bigoplus_{j \geq 0} \left(\bigoplus_{0 \leq i \leq j} P_{-j}(-i)\right)$, where $N$ on $\bigoplus_{0 \leq i \leq j} P_{-j}(-i)$ is given by shifting down summands, and the Frobenius action on $P_{-j}$ is $\iota$-pure of weight $\lambda - j$. That is, $V$ is \textit{$\iota$-geometric}, by which we mean geometric in the sense of \cite[Definition 2.43, Definition 2.45]{MISC:4} with pure replaced by $\iota$-pure and with the assumption of Frobenius-semisimplicity removed. Conversely, when the Weil-Deligne representation corresponding to $V \in \mathcal{A}$ is $\iota$-geometric, the $l$-adic representation $V$ is a $\iota$-pure monodromy-weight structure.
\item[(3)] Let $X$ be a smooth projective scheme over the local field $K$ and let $X_{\overline{K}}$ denote its base change to the separable closure $\overline{K}$. Rephrased in our terminology, the monodromy-weight conjecture \cite[8., Principe 8.1]{ARTICLE:9} presumes that the $l$-adic étale cohomology group $H^n_{\text{ét}}\!\left(X_{\overline{K}},\mathbb{Q}_l\right) \in \mathcal{A}$ is a $\iota$-pure monodromy-weight structure of weight $n$, according to \ref{prop:3} equivalently semistable of slope $n$ with respect to the monodromy-slope function, for all $\iota$. \\
Scholl \cite{MISC:5} proposes that this conjecture would imply that for arbitrary finite type schemes over $K$, one gets mixed structures.
\end{itemize}
\end{subremark}

\begin{appendices}
\renewcommand{\thesection}{\Alph{section}.\!\!}

\section{Categories of filtered objects}

\subsection{Strict (co-)compatibility and the modular law}

Let again $\mathcal{A}$ denote an arbitrary quasi-abelian category. For a morphism in $W\mathcal{A}^\text{pre}$ whose underlying morphism in $\mathcal{A}$ is strict, strict compatibility with the filtrations is equivalent to strict co-compatibility (\ref{rmk:1}). In this section, we show that whether or not they coincide for arbitrary morphisms in $W\mathcal{A}^\text{pre}$ is linked to the validity or non-validity of the modular law (\ref{def:2}) in $\mathcal{A}$. \\

Quasi-abelianess of $\mathcal{A}$ ensures that the following definition can be formulated without ambiguity.

\begin{subdefinition}
\namedlabel{def:2}{Definition \thesubsubsection}
We say that the \textit{modular law} holds in the quasi-abelian category $\mathcal{A}$ if for all strict subobjects $A \hookrightarrow B \hookrightarrow X$ and $C \hookrightarrow X$ in $\mathcal{A}$, the canonical morphism
\begin{equation*}
A + (B \cap C) \hookrightarrow B \cap (A + C)
\end{equation*}
is an isomorphism.
\end{subdefinition}

\begin{subexample}
\namedlabel{ex:8}{Example \thesubsubsection}
The modular law is a feature of abelian categories and is valid also when $\mathcal{A}$ is the category of finitely filtered objects of an abelian category. To give another example, the modular law holds in the category of finitely generated torsionfree modules over a Noetherian domain $R$, for one has a strongly exact, faithful functor to the category of finite-dimensional vector spaces over the fraction field of $R$. In contrast, the modular law fails to hold in the category of Hausdorff topological abelian groups, as is demonstrated in \cite[Example 2.9]{ARTICLE:4} and by other means in \ref{ex:2} below.
\end{subexample}

\begin{sublemma}
\namedlabel{lemma:2}{Lemma \thesubsubsection}
For a quasi-abelian category $\mathcal{A}$, the following are equivalent:
\begin{itemize}
\item[(1)] The modular law holds in $\mathcal{A}$.
\item[(2)] For all epi-monics $f:M \to N$ in $\mathcal{A}$ and all strict subobjects $P \hookrightarrow M$, the canonical morphism $P \hookrightarrow f^{-1}\!\left(f(P)\right)$ is an isomorphism.
\item[(3)] For all epi-monics $f: M \to N$ in $\mathcal{A}$ and all strict subobjects $Q \hookrightarrow N$, the canonical morphism $f\!\left(f^{-1}(Q)\right) \hookrightarrow Q$ is an isomorphism.
\end{itemize}
In this case, for any morphism $f: M \to N$ in $\mathcal{A}$ and strict subobjects $P \hookrightarrow M$, $Q \hookrightarrow N$, the morphisms
\begin{equation*}
P + f^{-1}(Q) \hookrightarrow f^{-1}\!\left(f(P) + Q\right)
\end{equation*}
and
\begin{equation*}
f\!\left(f^{-1}(Q) \cap P\right) \hookrightarrow Q \cap f(P)
\end{equation*}
are isomorphisms.
\end{sublemma}

\begin{proof}
We first remark that it does not require the modular law in order that for a strict morphism $f: M \to N$ in $\mathcal{A}$ and strict subobjects $P \hookrightarrow M$, $Q \hookrightarrow N$, the canonical morphisms
\begin{equation*}
P + \text{Ker}f \hookrightarrow f^{-1}\!\left(f(P)\right)
\end{equation*}
and
\begin{equation*}
f\!\left(f^{-1}(Q)\right) \hookrightarrow Q \cap f(M)
\end{equation*}
are isomorphisms. Hence, as one verifies by a straightforward computation, these identities hold for arbitrary morphisms in $\mathcal{A}$ if and only if they hold for epi-monics in $\mathcal{A}$, which according to our claim is equivalent to the validity of the modular law in $\mathcal{A}$. So if the modular law holds, then for an arbitrary morphism $f$, the claimed isomorphisms give rise to the chain of identities
\begin{equation*}
\begin{split}
f^{-1}\!\left(f(P) + Q\right) &= f^{-1}\!\left(\left(f(P) + Q\right) \cap f(M)\right) = f^{-1}\!\left(f(P) + \left(Q \cap f(M)\right)\right) \\
&= f^{-1}\!\left(f(P) + f\!\left(f^{-1}(Q)\right)\right) = f^{-1}\!\left(f\!\left(P + f^{-1}(Q)\right)\right) \\
&= \left(P + f^{-1}(Q)\right) + \text{Ker}f = P + f^{-1}(Q)
\end{split}
\end{equation*}
and, dually (see below), $f\!\left(f^{-1}(Q) \cap P\right) = Q \cap f(P)$. \\
For the proof of the claim, we will exploit duality. Namely, in the opposite category $\mathcal{A}^\circ$ of $\mathcal{A}$, we have $\left(X/(A \cap B)\right)^\circ = (X/A)^\circ + (X/B)^\circ$ and $\left(X/(A + B)\right)^\circ = (X/A)^\circ \cap (X/B)^\circ$. Consequently, the modular law holds in $\mathcal{A}$ if and only if it holds in $\mathcal{A}^\circ$. Moreover, $f\!\left(f^{-1}(Q)\right) \hookrightarrow Q$ is an isomorphism if and only if $Y/f\!\left(f^{-1}(Q)\right) \twoheadrightarrow Y/Q$ is an isomorphism, and the opposite of the latter is the morphism $(Y/Q)^\circ \hookrightarrow \left(f^\circ\right)^{-1}\!\left(f^\circ\!\left((Y/Q)^\circ\right)\right)$ in $\mathcal{A}^\circ$. It therefore suffices to show the equivalence between (1) and (2). \\
Suppose (2). Given strict subobjects $A \hookrightarrow B \hookrightarrow X$ and $C \hookrightarrow X$, we may consider the epi-monic
\begin{equation*}
f: B/(B \cap C) = \text{Coim}\!\left(B \to X/C\right) \to \text{Im}\!\left(B \to X/C\right) = (B + C)/C
\end{equation*}
and the strict subobject $\left(A + (B \cap C)\right)/(B \cap C) \hookrightarrow B/(B \cap C)$, for which we know that
\begin{equation*}
\begin{split}
\left(A + (B \cap C)\right)/(B \cap C) &= f^{-1}\!\left(f\!\left(\left(A + (B \cap C)\right)/(B \cap C)\right)\right) \\
&= f^{-1}\!\left((A + C)/C\right) = \text{Ker}\!\left(B/(B \cap C) \to (B + C)/(A + C)\right) \\
&= \left(B \cap (A + C)\right)/(B \cap C).
\end{split}
\end{equation*}
We deduce that the modular law is valid in $\mathcal{A}$, that is (1) holds. \\
Conversely, assume (1). We claim that every epi-monic in $\mathcal{A}$ is of the form considered in the previous paragraph. Indeed, let $g: M \to N$ be any morphism in $\mathcal{A}$. Then $g$ factors as $M \stackrel{(\text{id},g)}{\to} M \oplus N \twoheadrightarrow N$, where the second arrow is the canonical one. Now $(\text{id},g)$ is strict monic as the kernel of $M \oplus N \stackrel{(-g,\text{id})}{\to} N$, and if we put $(B \hookrightarrow X) \defeq (\text{id},g)$, $(C \hookrightarrow X) \defeq \left(M \hookrightarrow M \oplus N\right)$, we have $g = \left(B \to X/C\right)$ and further $\left(\text{Coim}g \to \text{Im}g\right) = B/(B \cap C) \to (B + C)/C$ as claimed. By \ref{rmk:6}(1), we can write every strict subobject of $B/(B \cap C)$ in the form $A/(B \cap C)$ for some $B \cap C \hookrightarrow A \hookrightarrow B$. From this, we see that the validity of the modular law implies (2).
\end{proof}

\begin{subexample}
\namedlabel{ex:2}{Example \thesubsubsection}
We revisit \ref{ex:8}. In the quasi-abelian category of Hausdorff topological abelian groups, consider the morphism $f: \mathbb{R}_\text{disc} \to \mathbb{R}$ between the real numbers with the discrete topology and the real numbers with their usual topology which is given by the identity on underlying sets. This morphism is epi-monic, however not an isomorphism. Now $\mathbb{Q}_\text{disc} \hookrightarrow \mathbb{R}_\text{disc}$ is a closed subgroup and thus a strict subobject, for which $f\!\left(\mathbb{Q}_\text{disc}\right)$ is the closure of the subgroup $\mathbb{Q} \subset \mathbb{R}$, hence $f\!\left(\mathbb{Q}_\text{disc}\right) = \mathbb{R}$. But then $f^{-1}\!\left(f\!\left(\mathbb{Q}_\text{disc}\right)\right) = \mathbb{R}_\text{disc}$. So by \ref{lemma:2}, the modular law does not hold in Hausdorff topological abelian groups.
\end{subexample}

\begin{intermediate}
Recall \ref{def:3}.
\end{intermediate}

\begin{subproposition}
\namedlabel{prop:4}{Proposition \thesubsubsection}
The modular law holds in the quasi-abelian category $\mathcal{A}$ if and only if for all epi-monics $f: M \to N$ in $W\mathcal{A}^\text{pre}$, we have that $f$ is strictly compatible with the filtrations precisely when it is strictly co-compatible with the filtrations.
\end{subproposition}

\begin{proof}
Let $f: M \to N$ be any morphism in $W\mathcal{A}^\text{pre}$. If the modular law holds in $\mathcal{A}$, then by \ref{lemma:2}, the morphisms $a$, $b$ in
\begin{equation*}
\begin{split}
&f\!\left(W_{\leq\lambda}M\right) = f\!\left(W_{\leq\lambda}M + \text{Ker}_\mathcal{A}f\right) \hookrightarrow f\!\left(f^{-1}\!\left(W_{\leq\lambda}N\right)\right) \stackrel{a}{\hookrightarrow} f(M) \cap W_{\leq\lambda}N \\
&W_{\leq\lambda}M + \text{Ker}_\mathcal{A}f \stackrel{b}{\hookrightarrow }f^{-1}\!\left(f\!\left(W_{\leq\lambda}M\right)\right) \hookrightarrow f^{-1}\!\left(f(M) \cap W_{\leq\lambda}N\right) = f^{-1}\!\left(W_{\leq\lambda}N\right)
\end{split}
\end{equation*}
are isomorphisms for all $\lambda$. Therefore, if the modular law holds in $\mathcal{A}$, $f$ is strictly compatible with the filtrations precisely when it is strictly co-compatible with the filtrations. \\
Conversely, suppose that strict compatibility is equivalent to strict co-compatibility for epi-monics in $W\mathcal{A}^\text{pre}$. Let $f: M \to N$ be an epi-monic in $\mathcal{A}$ and let $P \hookrightarrow M$, $Q \hookrightarrow N$ be strict subobjects. According to \ref{lemma:2}, in order to verify that the modular law holds in $\mathcal{A}$, we need to show that $P = f^{-1}\!\left(f(P)\right)$ and $f\!\left(f^{-1}(Q)\right) = Q$. But
\[\begin{tikzcd}
{\alpha:} & M & N & {\beta:} & M & N \\
& P & {f(P)} && {f^{-1}(Q)} & Q \\
& 0 & 0 && 0 & 0
\arrow["f", from=1-2, to=1-3]
\arrow["f", from=1-5, to=1-6]
\arrow[hook', from=2-2, to=1-2]
\arrow[from=2-2, to=2-3]
\arrow[hook', from=2-3, to=1-3]
\arrow[hook', from=2-5, to=1-5]
\arrow[from=2-5, to=2-6]
\arrow[hook', from=2-6, to=1-6]
\arrow[hook', from=3-2, to=2-2]
\arrow[equals, from=3-2, to=3-3]
\arrow[hook', from=3-3, to=2-3]
\arrow[hook', from=3-5, to=2-5]
\arrow[equals, from=3-5, to=3-6]
\arrow[hook', from=3-6, to=2-6]
\end{tikzcd}\]
are epi-monics in $W\mathcal{A}^\text{pre}$ and $\alpha$ is strictly compatible with the filtrations, hence by assumption $\alpha$ is also strictly co-compatible with the filtrations, which amounts to $P = f^{-1}\!\left(f(P)\right)$. Analogously, $\beta$ is strictly co-compatible with the filtrations, so by assumption also strictly compatible, which means that $f\!\left(f^{-1}(Q)\right) = Q$.
\end{proof}

\subsection{Quasi-abelianess}

Let $\mathcal{A}$ be a quasi-abelian category and let $W\mathcal{A}^\text{pre}$, as in \ref{def:1}, be the category of finite $\Lambda$-indexed filtrations on $\mathcal{A}$. We show that this category is quasi-abelian precisely when the modular law holds in $\mathcal{A}$. In contrast, $W\mathcal{A}^\text{pre}$ is abelian only if $\mathcal{A}$ is trivial.

\begin{subproposition}
\namedlabel{prop:1}{Proposition \thesubsubsection}
$W\mathcal{A}^\text{pre}$ is an additive category with kernels and cokernels. It is quasi-abelian if and only if the modular law holds in $\mathcal{A}$.
\end{subproposition}

\begin{proof}
It is clear that $W\mathcal{A}^\text{pre}$ has biproducts and kernels, which are preserved by the inclusion $W\mathcal{A}^\text{pre} \hookrightarrow \text{Fun}(\Lambda,\mathcal{A})$. To get additivity, one checks that if $f,g: M \to N$ are morphisms in $W\mathcal{A}^\text{pre}$, their difference $f-g$ in $\mathcal{A}$ respects the filtrations. For the cokernel of $f$ in $W\mathcal{A}^\text{pre}$, the formula $W_{\leq\lambda}\text{Coker}f = \text{Im}_\mathcal{A}\!\left(W_{\leq\lambda}N \to \text{Coker}_\mathcal{A}f\right) = \left(W_{\leq\lambda}N + f(M)\right)/f(M)$ holds. If $\mathcal{A}$ is abelian, the epi-monic $\text{Coim}_\mathcal{A}\!\left(W_{\leq\lambda}N \to \text{Coker}_\mathcal{A}f\right) = W_{\leq\lambda}N/\left(W_{\leq\lambda}N \cap f(M)\right) \to W_{\leq\lambda}\text{Coker}f$ is an isomorphism, and a proof of quasi-abelianess of $W\mathcal{A}^\text{pre}$ in this case is given in \cite[Theorem 3.9]{MISC:1}. \\
Suppose that the modular law holds in the quasi-abelian category $\mathcal{A}$. In order to verify that then $W\mathcal{A}^\text{pre}$ is quasi-abelian, it suffices to show that strict monics are stable under pushout as the dual statement follows from \ref{lemma:1}. \\
Given a pushout square in $W\mathcal{A}^\text{pre}$ as displayed on the left with $\iota: M' \hookrightarrow M$ strict monic, one can form the pushout of the same diagram in $\text{Fun}(\Lambda,\mathcal{A})$ to obtain the square displayed on the right with $\kappa: N' \hookrightarrow N''$ strict monic, and it holds that $\text{colim}~N''= N$ in $\mathcal{A}$ as well as $W_{\leq\lambda}N = \text{Im}_\mathcal{A}\!\left(N''(\lambda) \to N\right)$ for all $\lambda$:
\[\begin{tikzcd}
{M'} & M & {M'} & M \\
{N'} & N & {N'} & {N''}
\arrow["\iota", hook, from=1-1, to=1-2]
\arrow[from=1-1, to=2-1]
\arrow[from=1-2, to=2-2]
\arrow["\iota", hook, from=1-3, to=1-4]
\arrow[from=1-3, to=2-3]
\arrow[from=1-4, to=2-4]
\arrow["\sigma"', from=2-1, to=2-2]
\arrow["\lrcorner"{anchor=center, pos=0.125, rotate=180}, draw=none, from=2-2, to=1-1]
\arrow["\kappa"', hook, from=2-3, to=2-4]
\arrow["\lrcorner"{anchor=center, pos=0.125, rotate=180}, draw=none, from=2-4, to=1-3]
\end{tikzcd}\]
Let $Q$ denote the cokernel of $\kappa$ in $\text{Fun}(\Lambda,\mathcal{A})$. It is isomorphic to the cokernel $P$ of $\iota$ in $\text{Fun}(\Lambda,\mathcal{A})$, and as $\iota$ is strict monic in $W\mathcal{A}^\text{pre}$, the morphisms $Q(\lambda) \to Q$ in $\mathcal{A}$ are at least monic. It follows that $W_{\leq\lambda}N' = \text{Ker}_\mathcal{A}\!\left(N''(\lambda) \to Q(\lambda)\right) = \text{Ker}_\mathcal{A}\!\left(N''(\lambda) \to Q\right)$, so that $W_{\leq\lambda}N' \hookrightarrow N''(\lambda)$ is the pullback in $\mathcal{A}$ of $N' \hookrightarrow N$ along $N''(\lambda) \to N$. We obtain the following pullback squares in $\mathcal{A}$ with $f$ epi-monic:
\[\begin{tikzcd}
{W_{\leq\lambda}N'} & {f^{-1}\!\left(N' \cap W_{\leq\lambda}N\right)} & {N' \cap W_{\leq\lambda}N} \\
{N''(\lambda)} & {\text{Coim}_\mathcal{A}\!\left(N''(\lambda) \to N\right)} & {W_{\leq\lambda}N}
\arrow["\cong", from=1-1, to=1-2]
\arrow[hook, from=1-1, to=2-1]
\arrow["\lrcorner"{anchor=center, pos=0.125}, draw=none, from=1-1, to=2-2]
\arrow["g", hook, from=1-2, to=1-3]
\arrow[hook, from=1-2, to=2-2]
\arrow["\lrcorner"{anchor=center, pos=0.125}, draw=none, from=1-2, to=2-3]
\arrow[hook, from=1-3, to=2-3]
\arrow[two heads, from=2-1, to=2-2]
\arrow["f"', from=2-2, to=2-3]
\end{tikzcd}\]
According to \ref{lemma:2}, $g$ is an isomorphism when the modular law holds in $\mathcal{A}$. Then we have $W_{\leq\lambda}N' = N' \cap W_{\leq\lambda}N$, which means that $\sigma$ is strict monic in $W\mathcal{A}^\text{pre}$. \\
\ref{ex:7} below shows that the validity of the modular law in $\mathcal{A}$ is not only a sufficient, but also a necessary condition for $W\mathcal{A}^\text{pre}$ to be quasi-abelian. This completes the proof of the proposition.
\end{proof}

\begin{subexample}
\namedlabel{ex:7}{Example \thesubsubsection}
Suppose $A \hookrightarrow B \hookrightarrow X$ and $C \hookrightarrow X$ are strict subobjects in a quasi-abelian category $\mathcal{A}$ such that $B \cap C \hookrightarrow X$ factors through $A \hookrightarrow X$ and such that $A \hookrightarrow B \cap (A + C)$ is not an isomorphism. The following is a pushout in $W\mathcal{A}^\text{pre}$:
\[\begin{tikzcd}
{(B \cap C \hookrightarrow B)} & {(C \hookrightarrow X)} \\
{(A \hookrightarrow B)} & {(A + C \hookrightarrow X)}
\arrow[hook, from=1-1, to=1-2]
\arrow[from=1-1, to=2-1]
\arrow[from=1-2, to=2-2]
\arrow[from=2-1, to=2-2]
\arrow["\lrcorner"{anchor=center, pos=0.125, rotate=180}, draw=none, from=2-2, to=1-1]
\end{tikzcd}\]
By design, the bottom map is not strict. \\
When the modular does not hold in $\mathcal{A}$, one can always find strict subobjects as above: If it fails for $A' \hookrightarrow B \hookrightarrow X$, $C \hookrightarrow X$, replace $A'$ by $A \defeq A' + (B \cap C)$. An example is \cite[Example 2.9]{ARTICLE:4} showing that the modular law does not hold in the category of Hausdorff topological abelian groups, where one takes $X \defeq \mathbb{R}$, $A \defeq 2\mathbb{Z}$, $B \defeq \mathbb{Z}$ and $C \defeq \sqrt{2}\mathbb{Z}$.
\end{subexample}

\begin{subremark}
\namedlabel{rmk:1}{Remark \thesubsubsection}
Except in the trivial case, the category $W\mathcal{A}^\text{pre}$ is not abelian. Indeed, as soon as there is a nonzero object $M$ in the quasi-abelian category $\mathcal{A}$, the morphism of filtered objects
\[\begin{tikzcd}
M & M \\
0 & M
\arrow[equals, from=1-1, to=1-2]
\arrow[hook', from=2-1, to=1-1]
\arrow[hook, from=2-1, to=2-2]
\arrow[equals, from=2-2, to=1-2]
\end{tikzcd}\]
is an example of an epi-monic in $W\mathcal{A}^\text{pre}$ which is not an isomorphism. \\
We analyse when a morphism $f: M \to N$ in $W\mathcal{A}^\text{pre}$ is strict. On the one hand, $f$ is strict if and only if the monic
\begin{equation*}
\varphi_{\leq\lambda}: W_{\leq\lambda}\text{Coim}f = \frac{W_{\leq\lambda}M + \text{Ker}_\mathcal{A}f}{\text{Ker}_\mathcal{A}f} \to W_{\leq\lambda}N \cap f(M) = W_{\leq\lambda}\text{Im}f
\end{equation*}
is an isomorphism for all $\lambda$. On the other hand, $f$ is strictly compatible with the filtrations (\ref{def:3}) if and only if $\varphi_{\leq\lambda}$ is epi for all $\lambda$. When the underlying morphism of $f$ in $\mathcal{A}$ is strict, $\varphi_{\leq\lambda}$ is even strict monic and hence it is an isomorphism as soon as it is epi. So in this situation, $f$ is strict in $W\mathcal{A}^\text{pre}$ if and only if it is strictly compatible with the filtrations, if and only if (as one shows by duality or by hand) it is strictly co-compatible with the filtrations.
\end{subremark}

\subsection{The quasi-tannakian case}
\namedlabel{app:A.3}{Appendix \thesubsection}

In this section, consider a \textit{quasi-tannakian category} \cite[Definition 2.1.1]{ARTICLE:1} $\mathcal{A}$, that is a rigid symmetric monoidal, quasi-abelian category, linear over $k = \text{End}(\mathbf{1})$ ($\mathbf{1}$ the tensor unit), such that $k$ is a field of characteristic zero, and which admits an exact, faithful, $k$-linear tensor functor $\omega$ (the fiber functor) to the category of finite-dimensional vector spaces over some extension field $K$ of $k$. It follows from \cite[Lemma 2.1.5(2)]{ARTICLE:1} that $\omega$ is automatically strongly exact. \\
For every $M \in \mathcal{A}$, we fix a choice of the essentially unique dual $M^\lor$, the evaluation map $\text{ev}_M: M \otimes M^\lor \to \mathbf{1}$ and the co-evaluation map $\text{coev}_M: \mathbf{1} \to M^\lor \otimes M$. The assumption on the characteristic of $k$ ensures that the quasi-tannakian rank, which is defined by $\text{rk}(M) \defeq \left(\text{ev}_M \circ b_{M^\lor\!,M} \circ \text{coev}_M\right) \in \text{End}(\mathbf{1})$ ($b$ the braiding), is a rank function on $\mathcal{A}$. Indeed, the formula $\omega\!\left(\text{rk}(M)\right)\!\left(1 \in K\right) = \text{dim}_K\!\left(\omega(M)\right)$ holds. \\

When in the definition of a quasi-tannakian category, $\mathcal{A}$ is abelian, one recovers one of the characterisations of a tannakian category. \\

Here, for an arbitrary quasi-tannakian category $\mathcal{A}$, we construct a quasi-tannakian structure on the category $W\mathcal{A}^\text{pre}$ and show that, when $\mathcal{A}$ is endowed with a $\otimes$-multiplicative slope function for the quasi-tannakian rank, the category $W\mathcal{A}$ becomes tannakian.

\begin{subconstruction}
\namedlabel{cstr:2}{Construction \thesubsubsection}
$\text{}$
\begin{itemize}
\item[(1)] $W\mathcal{A}^\text{pre}$ inherits the structure of a symmetric monoidal category with tensor product
\begin{equation*}
(M,W) \otimes \left(M',W'\right) \defeq \left(M \otimes M', \left(\sum_{\nu + \nu' = \lambda} W_{\leq\nu}M \otimes W'_{\leq\nu'}M'\right)_\lambda\right)
\end{equation*}
and tensor unit $\mathbf{1}'$ given by the tensor unit $\mathbf{1}$ of $\mathcal{A}$ equipped with the unique filtration which has a single break at 0. Note that the associator, the unitors and the braiding in $\mathcal{A}$ respect all filtrations by naturality, and that the $\text{End}\!\left(\mathbf{1}'\right) = k$-linear structure on $\text{Hom}_{W\mathcal{A}^\text{pre}}\!\left((M,W),\left(M',W'\right)\right)$ is the structure of sub-vector space of $\text{Hom}_\mathcal{A}\!\left(M,M'\right)$.  \\
Given $(M,W) \in W\mathcal{A}^\text{pre}$, put
\begin{equation*}
(M,W)^\lor \defeq \left(M^\lor,\left(\left(M/W_{<-\lambda}M\right)^\lor\right)_\lambda\right) \in W\mathcal{A}^\text{pre}.
\end{equation*}
In \ref{lemma:4}(1),(3) below, we will verify that, with respect to this filtration on $M^\lor$, the evaluation map $\text{ev}_M: M \otimes M^\lor \to \mathbf{1}$ and the coevaluation map $\text{coev}_M: \mathbf{1} \to M^\lor \otimes M$ are morphisms of filtered objects, so that $(M,W)^\lor$ is a dual for $(M,W)$ in $W\mathcal{A}^\text{pre}$. Hence $W\mathcal{A}^\text{pre}$ is rigid. \\
Furthermore, $W\mathcal{A}^\text{pre}$ admits a fiber functor given by the composition of the fiber functor $\omega$ of $\mathcal{A}$ with the exact, faithful, linear tensor functor $W\mathcal{A}^\text{pre} \to \mathcal{A}$. \\
Since $\omega$ is strongly exact and reflects epis, the modular law holds in $\mathcal{A}$. So according to \ref{prop:1}, $W\mathcal{A}^\text{pre}$ is quasi-abelian. In total, $W\mathcal{A}^\text{pre}$ is again a quasi-tannakian category.
\item[(2)] If the quasi-tannakian category $\mathcal{A}$ is endowed with a slope function (with respect to the quasi-tannakian rank) which is \textit{$\otimes$-multiplicative} \cite [Definition 2.3.1]{ARTICLE:1}, then $W\mathcal{A}$ becomes a tannakian category. \\
Concretely, $\otimes$-multiplicativity means that if $M_1$, $M_2$ are semistable of slopes $\lambda_1$, $\lambda_2$, then $M_1 \otimes M_2$ is semistable of slope $\lambda_1 + \lambda_2$. As a consequence, $W\mathcal{A} \subset W\mathcal{A}^\text{pre}$ is closed under the tensor product and contains the tensor unit: The former is true because, as we will show in \ref{lemma:4}(2), the associated graded $\text{gr}: W\mathcal{A}^\text{pre} \to \mathcal{A}$ is a tensor functor. The latter amounts to $\mathbf{1} \in \mathcal{A}$ being semistable of slope 0. $\mu(\mathbf{1}) = 0$ follows from \cite[Theorem 2.3.3(2b), Theorem 2.2.9(1)]{ARTICLE:1} and if $i: U \hookrightarrow \mathbf{1}$ is the universal destabilising subobject, the strict monic $U \otimes i: U \otimes U \hookrightarrow U$ shows that $\mu(U) \leq 0$, hence $\mathbf{1} \in \mathcal{A}(0)$. \\
For rigidity, it suffices to check that $W\mathcal{A}$ is closed under taking duals in $W\mathcal{A}^\text{pre}$. Indeed, for $(M,W) \in W\mathcal{A}^\text{pre}$, we know that $\text{gr}_\lambda\!\left((M,W)^\lor\right) = \left(\text{gr}^W_{-\lambda}M\right)^\lor$ and deduce from \cite[Theorem 2.3.3(2b), Theorem 2.2.9(2)]{ARTICLE:1} that $(M,W) \in W\mathcal{A}$ if and only if $(M,W)^\lor \in W\mathcal{A}$.
\end{itemize}
\end{subconstruction}

\begin{intermediate}
We fill in the proof of the claims from \ref{cstr:2}.
\end{intermediate}

\begin{sublemma}
\namedlabel{lemma:4}{Lemma \thesubsubsection}
Let $\mathcal{A}$ be a rigid symmetric monoidal, quasi-abelian category. Equip $W\mathcal{A}^\text{pre}$ with the structure of a symmetric monoidal category from \ref{cstr:2}(1), and given $(M,W) \in W\mathcal{A}^\text{pre}$, define a filtration on the dual $M^\lor$ in $\mathcal{A}$ by $W_{\leq\lambda}\!\left(M^\lor\right) \defeq \left(M/W_{<-\lambda}M\right)^\lor$.
\begin{itemize}
\item[(1)] The evaluation map $\text{ev}_M: M \otimes M^\lor \to \mathbf{1}$ in $\mathcal{A}$ is a morphism of filtered objects. 
\item[(2)] The associated graded $\text{gr}:W\mathcal{A}^\text{pre} \to \mathcal{A}$ is a tensor functor.
\item[(3)] Also the coevaluation map $\text{coev}_M: \mathbf{1} \to M^\lor \otimes M$ respects the filtrations.
\end{itemize}
\end{sublemma}

\begin{proof}
For (1), we need to show that the restriction of the evaluation map to
\begin{equation*}
W_{<0}\!\left(M \otimes M^\lor\right) = \sum_{\lambda \in \Lambda} W_{\leq\lambda}M \otimes \left(M/W_{\leq \lambda}M\right)^\lor
\end{equation*}
is zero. By adjunction and by construction of the action of $(-)^\lor$ on morphisms, the bijections $\text{Hom}_\mathcal{A}(X \otimes Y,Z) \cong \text{Hom}_\mathcal{A}\!\left(Y,X^\lor \otimes Z\right)$ are functorial in all three arguments. Therefore, for every $\lambda \in \Lambda$, the following diagram with the obvious maps commutes:
\[\begin{tikzcd}
{\text{Hom}_\mathcal{A}\left(M^\lor,M^\lor\right)} & {\text{Hom}_\mathcal{A}\!\left(M \otimes M^\lor,\mathbf{1}\right)} \\
{\text{Hom}_\mathcal{A}\!\left(M^\lor,\left(W_{\leq\lambda}M\right)^\lor\right)} & {\text{Hom}_\mathcal{A}\!\left(W_{\leq\lambda}M \otimes M^\lor,\mathbf{1}\right)} \\
{\text{Hom}_\mathcal{A}\!\left(\left(M/W_{\leq\lambda}M\right)^\lor,\left(W_{\leq\lambda}M\right)^\lor\right)} & {\text{Hom}_\mathcal{A}\!\left(W_{\leq\lambda}M \otimes \left(M/W_{\leq\lambda}M\right)^\lor,\mathbf{1}\right)}
\arrow["\cong", from=1-1, to=1-2]
\arrow[from=1-1, to=2-1]
\arrow[from=1-2, to=2-2]
\arrow["\cong", from=2-1, to=2-2]
\arrow[from=2-1, to=3-1]
\arrow[from=2-2, to=3-2]
\arrow["\cong", from=3-1, to=3-2]
\end{tikzcd}\]
Along the outer square, $\text{id}_{M^\lor}$ gets mapped as follows:
\[\begin{tikzcd}[column sep=small]
{\text{id}_{M^\lor}} && {\text{ev}_M} \\
{\left(W_{\leq\lambda}M \hookrightarrow M \twoheadrightarrow M/W_{\leq\lambda}M\right)^\lor = 0} & 0 & {\text{ev}_M \circ \left(W_{\leq\lambda}M \otimes \left(M/W_{\leq\lambda}M\right)^\lor \hookrightarrow M \otimes M^\lor\right)}
\arrow[maps to, from=1-1, to=1-3]
\arrow[maps to, from=1-1, to=2-1]
\arrow[maps to, from=1-3, to=2-3]
\arrow[maps to, from=2-1, to=2-2]
\end{tikzcd}\]
This establishes (1). \\
Let us turn to (2) and (3). For any other $(M',W') \in W\mathcal{A}^\text{pre}$ and the filtrations on the duals defined above, the canonical morphism $\varphi_{M,M'}: (M')^\lor \otimes M^\lor \to (M \otimes M')^\lor$ in $\mathcal{A}$ respects the filtrations. Let $\varphi_{(M,W),(M',W')}$ denote the resulting morphism in $W\mathcal{A}^\text{pre}$. In $\mathcal{A}$, $\varphi_{M,M'}$ is an isomorphism. Furthermore, one has functorial isomorphisms $\text{gr}_\lambda\!\left((M,W)^\lor\right) \stackrel{\cong}{\to} \left(\text{gr}^W_{-\lambda}M\right)^\lor$ for all $\lambda \in \Lambda$, hence a functorial isomorphism $\hat{\text{gr}}_{(M,W)}: \text{gr}\!\left((M,W)^\lor\right) \stackrel{\cong}{\to} \left(\text{gr}(M,W)\right)^\lor$.  Also, since $\otimes$ is bi-exact on $\mathcal{A}$ due to rigidity, we get natural morphisms
\begin{equation*}
\bigoplus_\nu \text{gr}_{\nu}^W M \otimes \text{gr}_{\lambda-\nu}^{W'} M' \cong \bigoplus_\nu \frac{W_{\leq\nu}M \otimes W'_{\leq\lambda-\nu}M'}{W_{<\nu}M \otimes W'_{\leq\lambda-\nu}M' + W_{\leq\nu}M \otimes W'_{<\lambda-\nu}M'} \to \text{gr}_\lambda^{W \otimes W'}\!\left(M \otimes M'\right),
\end{equation*}
thus a natural morphism $\tilde{\text{gr}}_{(M,W),(M',W')}: \text{gr}(M,W) \otimes \text{gr}(M',W') \to \text{gr}(M \otimes M', W \otimes W')$. \\
We cannot yet apply \cite[Lemma A.9, Corollary A.11]{ARTICLE:1} to these transformations, but imitate the argument. The diagram
\[\begin{tikzcd}
{\text{gr}\!\left((M',W')^\lor\right) \otimes \text{gr}\!\left((M,W)^\lor\right)} && {\text{gr}\!\left((M',W')^\lor \otimes (M,W)^\lor\right)} \\
{\text{gr}(M',W')^\lor \otimes \text{gr}(M,W)^\lor} && {\text{gr}\!\left(\left((M,W) \otimes (M',W')\right)^\lor\right)} \\
{\left(\text{gr}(M,W) \otimes \text{gr}(M',W')\right)^\lor} && {\text{gr}\!\left((M,W) \otimes (M',W')\right)^\lor}
\arrow["{\tilde{\text{gr}}_{(M',W')^\lor,(M,W)^\lor}}", from=1-1, to=1-3]
\arrow["{\hat{\text{gr}}_{(M',W')} \otimes \hat{\text{gr}}_{(M,W)}}"', "\cong", from=1-1, to=2-1]
\arrow["{\text{gr}(\varphi_{(M,W),(M',W')})}", from=1-3, to=2-3]
\arrow["{\varphi_{\text{gr}(M,W),\text{gr}(M',W')}}"', "\cong", from=2-1, to=3-1]
\arrow["{\hat{\text{gr}}_{(M,W) \otimes (M',W')}}", "\cong"', from=2-3, to=3-3]
\arrow["{(\tilde{\text{gr}}_{(M,W),(M'W')})^\lor}", from=3-3, to=3-1]
\end{tikzcd}\]
commutes: We saw in (1) that the evaluation map $\text{ev}_M: M \otimes M^\lor \to \mathbf{1}$ in $\mathcal{A}$ is a morphism of filtered objects with respect to the filtration we put on $M^\lor$, and one checks that the transformation $\hat{\text{gr}}$ fits into the commutative diagram of \cite[A.4]{ARTICLE:1} (using that $\hat{\text{gr}}$ is induced from the morphisms $\left(\text{gr}^W_\lambda M \hookrightarrow M/W_{<\lambda}M\right)^\lor$), which together is sufficient for the proof of commutativity in \cite[Lemma A.9]{ARTICLE:1} to carry over. \\
So $f \defeq \tilde{\text{gr}}_{(M',W')^\lor,(M,W)^\lor}$ has a left inverse. It is clear that $f$ is epi, thus it is an isomorphism. As every object of $W\mathcal{A}^\text{pre}$ is isomorphic to $(M,W)^\lor$ for some $(M,W) \in W\mathcal{A}^\text{pre}$, we conclude that the transformation $\tilde{\text{gr}}$ is an isomorphism. \\
Consequently, the commutative diagram shows that the morphisms $\varphi_{(M,W),(M',W')}$ induce isomorphisms on the graded pieces, hence are isomorphisms in $W\mathcal{A}^\text{pre}$. Dualising the triangle identities, one finds that $(\left(\text{coev}_M\right)^\lor\!\varphi_{M^\lor\!,M},~\varphi_{M,M^\lor}^{-1}\left(\text{ev}_M\right)^\lor)$ is an evaluation-coevaluation pair for $M^\lor \in \mathcal{A}$. By (1), $\left(\text{coev}_M\right)^\lor\!\varphi_{M^\lor\!,M}$ respects the filtrations, and therefore so does $\text{coev}_M$. This completes the proof of (2) and (3).
\end{proof}

\begin{intermediate}
When $\mathcal{A}$ is quasi-tannakian, one can check directly that $(-)^\lor$ defines a tensor endofunctor on $W\mathcal{A}^\text{pre}$, because the strict monics $W_{\leq\lambda}\varphi$ become vector space isomorphisms after applying the faithful fiber functor (use bases adapted to the filtrations). \\

To conclude, we look at examples of \ref{cstr:2}.
\end{intermediate}

\begin{subexample}
\namedlabel{ex:6}{Example \thesubsubsection}
In \ref{ex:4}, let $\mathcal{B}$ be a tannakian category. The analogue of \ref{cstr:2}(1) for descending filtrations shows that then $\mathcal{B}_1$ is canonically a quasi-tannakian category, and the rank function from \ref{ex:4} is the quasi-tannakian rank.
\begin{itemize}
\item[(1)] Call an object of $\mathcal{B}_1$ \textit{quasi-simple} if it admits no nontrivial strict subobject. Necessarily, when $0 \neq (M,F) \in \mathcal{B}_1$ is quasi-simple, the filtration $F$ has precisely one break. Every rank one object is quasi-simple, and these form the group of isomorphism classes of invertible objects $\text{Pic}\!\left(\mathcal{B}_1\right)$. \\
By \cite[Lemma 2.2.1]{ARTICLE:1}, the quasi-tannakian determinant $\text{det}(M,F) = \bigwedge^{\text{rk}(M,F)}(M,F)$, defined as the image of the antisymmetrisation map on $(M,F)^{\otimes\text{rk}(M,F)}$, is a homomorphism $\text{det}: K_0\!\left(\mathcal{B}_1\right) \to \text{Pic}\!\left(\mathcal{B}_1\right)$. Moreover, here, the assignment of the unique break is a homomorphism $\delta: \text{Pic}\!\left(\mathcal{B}_1\right) \to \Lambda$. The formula $\delta\!\left(\text{det}(M,F)\right) = \sum_{\nu \in \Lambda} \text{rk}_\mathcal{B}\!\left(\text{gr}^\nu_F M\right)\nu$ holds when $(M,F)$ is quasi-simple, and since both sides are additive along short exact sequences in $\mathcal{B}_1$, it holds in general. So
\begin{equation*}
\mu(M,F) \defeq \frac{\delta\!\left(\text{det}(M,F)\right)}{\text{rk}_\mathcal{B}(M)} = \frac{\sum_{\nu \in \Lambda} \text{rk}_\mathcal{B}\!\left(\text{gr}^\nu_F M\right)\nu}{\text{rk}_\mathcal{B}(M)}
\end{equation*}
is the slope function on $\mathcal{B}_1$ from \ref{ex:4}, which is hence \textit{determinantal} \cite[Definition 2.2.6]{ARTICLE:1} in the present case. \\
In this setup, the verification of condition \eqref{eq:dagger} in \ref{ex:4}(1) simplifies, as \cite[Proposition 2.2.2(2),(3)]{ARTICLE:1} asserts that taking the determinant of a morphism between objects of the same rank preserves epi-monics and detects isomorphisms. This allows one to reduce to considering epi-monics between rank one objects $(M,F) \to (M',F')$ which fail to be isomorphisms. It is then clear that $\text{deg}(M,F) < \text{deg}(M',F')$. In other words, verifying condition \eqref{eq:dagger} amounts to showing that $\delta$ is strictly increasing with respect to the partial order on $\text{Pic}\!\left(\mathcal{B}_1\right)$ from \cite[Theorem 2.2.8]{ARTICLE:1} (while to get a slope function in the sense of \cite[Definition 1.3.1]{ARTICLE:1}, it suffices that $\delta$ be non-decreasing). \\
With the description of the semistable objects in \ref{ex:4}(3) already at hand, one sees that the slope function from \ref{ex:4} is even $\otimes$-multiplicative. Therefore, that it is determinantal is also a consequence of \cite[Theorem 2.3.3(2b)]{ARTICLE:1}, and the homomorphism $\delta: \text{Pic}\!\left(\mathcal{B}_1\right) \to \Lambda$ is necessarily given by assigning the unique break.
\item[(2)] The equivalence $\mathcal{B}_2 \simeq \left(\mathcal{B}_1\right)_1$ respects the obvious tensor structures on both sides, and by \ref{cstr:2}(1), $\left(\mathcal{B}_1\right)_1$ is quasi-tannakian. Similar remarks as in (1) apply: Take for $\delta: \text{Pic}\!\left(\mathcal{B}_2\right) \to \Lambda$ the homomorphism assigning to a rank one object $0 \neq (M,F,\bar{F})$ the sum of the unique break of $F$ and the unique break of $\bar{F}$ to recover the slope function from \ref{ex:5}. It is $\otimes$-multiplicative.
\end{itemize}
\end{subexample}

\begin{subexample}
All examples of \ref{sec:4} are instances of \ref{cstr:2}(2) and thus tannakian: abstract mixed Hodge structures $W\mathcal{B}_2$ when $\mathcal{B}$ is a tannakian category, mixed twistor structures and mixed monodromy-weight structures. This is well-known in each particular context (for Hodge structures at least when $\mathcal{B}$ is the category of finite-dimensional vector spaces over a field, not necessarily of characteristic zero), where one argues as is demonstrated in \cite[Proposition 1.10]{ARTICLE:10} for the case of twistor structures.
\end{subexample}

\end{appendices}

\bibliography{weight_filtrations_via_slopes}
\bibliographystyle{alpha}

\end{document}